\documentclass[hidelinks,onefignum,onetabnum]{siamart220329}

\RequirePackage{fix-cm}

\usepackage{enumitem}
\setlistdepth{5}

\newlist{myEnumerate}{enumerate}{5}
\setlist[myEnumerate,1]{label=(\arabic*)}
\setlist[myEnumerate,2]{label=(\Roman*)}
\usepackage{amsmath,amssymb,bm}
\usepackage{color,xcolor}
\usepackage{graphicx,graphics}
\definecolor{refgreen}{rgb}{0,0.5,0}
\usepackage{booktabs} % for tables
\usepackage{mathtools} %for rcases

\usepackage{pgfplots}
\pgfplotsset{compat=newest}
\pgfplotsset{plot coordinates/math parser=false}
\newlength\figureheight
\newlength\figurewidth 
\usepackage{tikzscale}
\usepackage{epstopdf}
\usetikzlibrary{calc,angles,quotes}
\usetikzlibrary{decorations.markings}
\usepackage{tikz}
\usetikzlibrary{intersections}

\usepackage{scalefnt} %scale font of tikzpicutre

\def   \d {\hspace{1.5pt}\mathrm{d}}

\usepackage{hyperref}
\usepackage[normalem]{ulem} % for strikethrough \sout{...}

\usepackage{algorithm} %For algorithms
\usepackage{algorithmic}

\usepackage{diffcoeff}
\usepackage{bm} %big font for math symbols

\definecolor{refblue}{rgb}{0,0,0.75}
\definecolor{refblueb}{rgb}{0,0,1}
\definecolor{refgreen}{rgb}{0.13,0.55,0.13}
\definecolor{refred}{rgb}{1,0,0}

\hypersetup{
	colorlinks   = true, %Colours links instead of ugly boxes
	urlcolor     = refblue, %Colour for external hyperlinks
	linkcolor    = refblueb, %Colour of internal links
	citecolor   = refgreen %Colour of citations
}

\usepackage{stmaryrd} %Defining llbracket and rrbracken [[, ]]

\newcommand{\R}{\mathbb{R}}

\newcommand{\IntRef}{\mathcal{I}_{\text{ref}}^n} %Interpolation Operator for Refining/Coarsening
\DeclareMathOperator{\osc}{\textnormal{osc}}
\DeclareMathOperator{\TOL}{\texttt{TOL}}
\DeclareMathOperator{\ProjOneSided}{Pr_h}
\newcommand{\Rcoarse}{\mathcal{R}_{\textnormal{c}}}
\newcommand{\Ga}{\varGamma}
\newcommand{\baruhtau}{\overline{\vphantom{\bar u}u_{h,\tau}}}

\newcommand{\comesh}{n-1 \oplus n} %n \otimes (n-1)

\usepackage[toc,page]{appendix}

\usepackage[autostyle=true]{csquotes}

\graphicspath{ {./Figures/} }
\usepackage{graphicx}
\usepackage{subcaption}

\usepackage{multirow}

\newcommand{\ebk}{\color{black}}

\ifpdf
\DeclareGraphicsExtensions{.eps,.pdf,.png,.jpg}
\else
\DeclareGraphicsExtensions{.eps}
\fi

\newcommand{\creflastconjunction}{, and~}

\newsiamremark{remark}{Remark}
\newsiamremark{hypothesis}{Hypothesis}
\crefname{hypothesis}{Hypothesis}{Hypotheses}
\newsiamthm{claim}{Claim}

\headers{A posteriori error est.~for parabolic surface PDEs}{B.~Kov\'acs and M.~Lantelme}

\title{A posteriori error estimates for parabolic partial differential equations on stationary surfaces}

\author{Bal\'azs Kov\'acs$^\dagger$ \and
Michael Lantelme\thanks{Institute of Mathematics, Paderborn University, Paderborn} 
		(\email{lantelme@math.uni-paderborn.de})}

\usepackage{cancel}	

\usepackage{amsopn}

\ifpdf
\hypersetup{
  pdftitle={StationaryParabolicASFEM},
  pdfauthor={B. Kovács, M. Lantelme}
}
\fi

\begin{document}

\maketitle

\begin{abstract}
	This paper develops and discusses a residual-based a posteriori error estimator for parabolic surface partial differential equations on closed stationary surfaces. The full discretization uses the surface finite element method in space and the backward Euler method in time. The proposed error indicator bounds the error quantities globally in space from above and below, and globally in time from above and locally from below. Based on the derived error indicator, a space--time adaptive algorithm is proposed. Numerical experiments illustrate and complement the theory.
\end{abstract}

% REQUIRED

% REQUIRED

\section{Introduction}
In this paper, we develop and analyse an a posteriori error estimator for parabolic partial differential equations (PDEs) on closed stationary surfaces,  and numerically test a classical space--time adaptive algorithm using newest-vertex bisection, which is based on the proposed error estimators.  The method uses surface finite elements (surface FEM) in space and the implicit Euler method in time. A residual-based error indicator is derived. We prove that the error indicator bounds the error globally in space from above and below, and globally in time from above and locally from below,  i.e.,~the a posteriori error estimator is shown to be efficient and reliable.

The proposed a posteriori error indicator consists of four parts: spatial, temporal, coarsening, and geometric terms. The first two are similar to their counterparts from the Euclidean case (with some additional geometric approximations). The geometric indicator is specific for surface PDEs (see, e.g.,~\cite{AdaptiveFEMBeltrami,APostElliptic}), while the coarsening indicator is essential for time-dependent problems, and particularly challenging for parabolic surface PDEs. 
The presented a posteriori error analysis is a careful generalisation of the classical residual-based error analysis, while it constantly deals with the geometric approximations due to the non-conforming surface finite element discretization (i.e.,~$V_h \nsubseteq V$), and also fully incorporates the coarsening and geometric indicators. 
The analysis and computability of the coarsening indicator requires novel tools and careful techniques: It is essential to use smallest common refinements -- and corresponding interpolation operators, which help to compare functions on subsequent meshes (which may be refined and coarsened yielding distinct discrete surfaces). Although common refinements were used for PDEs in Euclidean domains, see \cite{AdaptiveAlgHeatEq}, the geometry again needs to be handled carefully, this is achieved via the $\theta$-argument of \cite{KLLP_2017} comparing different surface quantities.
Since subsequent meshes usually differ, we use a backward Euler full discretization based on the surface mesh at the current time-step. This requires us to define two linear time interpolations of the numerical solution: one being continuous in time, used in comparing the residual with the error; the other being discontinuous in time, but easy to compute for the error indicators. 
 
We present an adaptive algorithm based on the error indicators which controls both the temporal step size and the refinement/coarsening of spatial meshes.

 Time dependent  surface partial differential equations have various applications in numerous fields, including fluid dynamics \cite{Heikes1995,Gross2007}, ice formation \cite{Myers2004}, brain imaging \cite{Memoli2004}, spectral geometry \cite{Reuter2009}, tumour modelling \cite{Eyles2019ATM,King2021FreeBP},  phase-separation \cite{ElliottRanner,ElliottSales2024_CH}, pattern formation \cite{Amago}, flows on surfaces \cite{BKNV,PorrmannVoigt,ElliottSales2024_NavierStokes_CH},  and mean curvature flow \cite{Huisken1987TheVP,MCF,MCF_surgery}. 

In 1988, Dziuk \cite{Dziuk1988} laid the groundwork for surface finite elements.
For surveys on parabolic surface problems and surface finite elements we refer to, e.g.,~\cite{ESFEM2007,SFEM_for_Para,Dziuk2013FiniteEM}.

We now give a literature overview on \textit{adaptivity for elliptic equations on surfaces}:

The earliest methods were bound to surfaces with a global parametrization \cite{Apel2005}. Demlow and Dziuk \cite{AdaptiveFEMBeltrami} introduced the first adaptive \textit{finite element method} on general surfaces for the Laplace–Beltrami equation, based on an a posteriori error analysis. A key insight is that the error can be split into two parts, a "residual part" from the PDE and a "geometric part" due to the surface approximation. Careful control over the geometry of the surface is needed when aiming for a robust approximation, even though the geometric quantities (except the data approximation) are non-dominant. Later on Camacho and Demlow \cite{APostElliptic} determined an efficient and reliable $L^2$ and point-wise error estimate for elliptic surface PDEs.
 To deal with geometric driven error indicators Bonito et al.~\cite{Bonito2013AFEMFG} introduced an additional adaptive routine to guarantee the error estimates, and to further guarantee shape-regularity throughout refinements, which is non-trivial for surface PDEs \cite{AdaptiveFEMBeltrami}.  

A posteriori error analysis for elliptic surface PDEs was extended to \textit{finite volume methods} by Ju, Tian and Wang \cite{Ju2009} and Demlow and Olshanskii \cite{Demlow2012}; and to \textit{discontinuous Galerkin methods} by Dedner and Madhavan \cite{Dedner2014}.

 We highlight some references on adaptive methods for PDEs in \textit{Euclidean domains}: Early work \cite{Dupont1982MeshMF,Bieterman1982TheFE} developed adaptive strategies based on a posteriori error estimates. We follow the framework introduced by Verfürth \cite{Verfuehrt1996,Verfuerth2003} which provides systematic proofs of reliability and efficiency. For stronger norms, most notably to obtain optimal $L^2 (L^\infty)$ bounds, we refer to, e.g., \cite{eriksson95} using strong stability estimates, \cite{MakridakisNochetto2003,ell_reconstruct} using elliptic reconstruction. 
Adaptive algorithms were analysed, e.g., in \cite{chen2004feng_adap_flat,AdaptiveAlgHeatEq}, based on error indicators derived in the Verfürth framework, establish convergence for parabolic problems and employ the fundamental solve–estimate–mark–refine strategy of Dörfler \cite{Doerfler1996}. 

To our knowledge, a posteriori error analysis and adaptivity was not yet studied for \textit{parabolic problems on surfaces} in the literature.  Moreover, the approach of this work may serve as a blueprint for the a posteriori error analysis for other non-conforming discretisation methods.  

Our results expand on the theory of  a posteriori error estimates  for PDEs on stationary surfaces. Therefore, making a significant step towards adaptive algorithms for evolving surface PDEs and geometric surface flows.  The numerical simulation of the above mentioned problems, e.g.,~singularity resolution for mean curvature flow, would greatly benefit from an adaptive solution approach.  

The paper is organised as follows: In Section~\ref{ch:surfacePDE} the heat equation on stationary surfaces is introduced.  In Section~\ref{ch_full_discr} we briefly discuss shape regularity for adaptive surface FEM. Afterwards we show the spatial and temporal discretization where for the latter we introduce a refinement interpolant which is needed for computation of functions on adaptive time-dependent meshes.  By introducing the refinement interpolant two choices of time-interpolation are available and discussed. We further introduce the smallest common triangulation of two consecutive meshes, and an interpolation in the corresponding finite element space.  
Then in Section~\ref{section:ReliablePara} we state the key results and derive the error indicator. 
Section~\ref{section:proofs} contains most the proofs: the equivalence to the error and residual, residual decomposition, some geometric estimates, and the main estimates for the error indicators. 
In Section~\ref{adaptivity}  we introduce an adaptive algorithm based on  the derived indicator. 
Section~\ref{ch:impl} discusses the implementation of the indicators and an adaptive algorithm. 
Finally, in Section~\ref{ch:experiments} a set of illustrative numerical experiments are given, which show the asymptotic behaviour of the error and reasonable refining and coarsening. 

The implementation of the described algorithm is  based on the $\ell$FEM Matlab package \cite{ellFEM}, and is  accessible on \texttt{\url{https://git.uni-paderborn.de/lantelme/parabolic-stationary-asfem}}, which also contains all numerical experiments. 

\section{Heat equation on stationary surfaces}
\label{ch:surfacePDE}

Let us consider the surface heat equation
\begin{equation}
\label{eq:heat_strong}
\begin{aligned}
\partial_t u(x,t) - \varDelta_\Ga u (x,t) &= f(x,t) \quad &&\forall  (x,t) \in  \Ga \times (0,T_{\max}] \\
u(x,0) &= u^0 (x) \quad &&\forall x \in \Ga ,
\end{aligned}
\end{equation}
where $\Ga \subset \R^3$ is always assumed to be a two-dimensional, closed, sufficiently smooth (at least $C^2$), stationary surface,  whose principal curvatures and their derivatives are bounded,  further $f \in C(0,T_{\max};L^2(\Ga))$ is a given inhomogeneity, and $u^0 \in L^2 (\Ga)$ is a given initial value.

\textbf{Weak formulation.}
Using Green's formula on closed surfaces \cite[Theorem~2.14]{Dziuk2013FiniteEM} the weak formulation reads:
Find $u \colon \Ga \times [0,T_{\max}] \rightarrow \R$, with $u(\cdot,0) = u^0$, such that
\begin{equation}
\label{eq:weak_form}
(\partial_t u(\cdot,t),v)_{\Ga} + (\nabla_\Ga u (\cdot,t), \nabla_\Ga v)_{\Ga} = (f,v)_{\Ga}, \quad \text{for all } v \in H^1 (\Ga), 
\end{equation}
for almost every $t \in (0,T_{\max})$. We denote the $L^2$-scalar product on $\Ga$ by $(\cdot,\cdot)_{\Ga} := (\cdot,\cdot)_{L^2(\Ga)}$.

\subsection{Surface definitions and operators}
\label{section:surface_operators}
We use the same setting introduced in \cite{Dziuk1988,ESFEM2007,Dziuk2013FiniteEM}.
%Assume that $\Ga$ is two-dimensional hypersurface with the same regularity as above. 
The surface $\Ga$ is assumed to be  (at least a $C^2$ surface)  given as the zero level-set of a signed distance function $d \colon \R^{3} \rightarrow \R$. 
 We assume that $\Ga$ lies within an open set $U \subset \R^3$, and is represented by a signed distance function $d\colon U \to \R$ such that $\Ga = \{ x \in \R^3 \mid d(x) = 0\}$. We further let $U$ be a strip around $\Ga$ of sufficiently small width $\delta < \big( \max_{x \in \Ga}\{|\kappa_1(x)|,|\kappa_2(x)|\} \big)^{-1}$, where $\kappa_i(x)$ denotes the principal curvatures at a point $x$. 
The tangential (or surface) gradient of a scalar function $u\colon \Ga \to \R$ is given by, see, e.g.,~\cite[Section~2.1 and Definition~2.3]{Dziuk2013FiniteEM},
\begin{equation*}
	\nabla_{\Ga} u := P \nabla \overline{u} := \nabla \overline{u} - (\nabla \overline{u} \cdot \nu) \nu ,
%	\nabla_{\Ga} u := P \nabla \overline{u} := \nabla \overline{u} - (\nabla \overline{u} \cdot \nu) \nu ,
\end{equation*} 
where $\nu \colon \Ga \to \R^3$ is the outer unit normal field to $\Ga$, and $\overline{u}$ is an extension of $u$ into $U$. The tangential gradient is independent of the extension.

The Laplace--Beltrami operator is given by $\varDelta_\Ga u = \nabla_{\Ga} \cdot \nabla_{\Ga} u$, where the surface divergence of a vector field $w \colon \Ga \rightarrow \R^3$ is defined by $\nabla_{\Ga} \cdot w = \sum_{j=1}^3 (\nabla_{\Ga} w_j)_j$.
Finally, the extended Weingarten map is denoted by $\mathcal{A} = \nabla_\Ga \nu$.

 We note that, if some constructions are avoided, e.g.~using the Weingarten map, then slightly less regularity $C^{1,\alpha}$ on $\Ga$ is also sufficient, see, e.g., \cite{BonitoDemlowNochetto}.

\section{Spatial and temporal discretizations}
\label{ch_full_discr}
\subsection{Surface finite element method}
\label{section:SFEM}

We start by describing the linear surface finite element method, following \cite{Dziuk1988,Dziuk2013FiniteEM}, and paying particular attention to issues arising due to adaptivity, see \cite{AdaptiveFEMBeltrami,Bonito2013AFEMFG}.

\textbf{Discrete surface, NVB refinement, and shape-regularity.} 
The surface is approximated by an admissible triangulation $\mathcal{T}_h  \subset U$,  which means following \cite{Dziuk2013FiniteEM} we require that the triangulation is an interpolation of $\Ga$ (all nodes of $\Ga_h$ lie on $\Ga$),  shape-regular  (there is a constant $\varrho > 0$ such that $\frac{h_T}{r_T} \leq \varrho$ holds for all $T \in \Ga_h$, where $h_T$ is the maximal edge length of the triangle and $r_T$ is the radius of its inscribed circle) and not a double covering. We let $h = \max_T h_T$,  which must satisfy $h \leq h_0$ with a sufficiently small $h_0 > 0$ due to $\Ga_h \subset U$.

Preserving shape-regularity of triangulations of surfaces throughout refinement and coarsening is non-trivial due to lifting nodes, see, e.g., the discussion in  \cite[Section~2.2]{AdaptiveFEMBeltrami}. In comparison, for the flat case appropriate refinement algorithms guarantee shape-regularity. 
We briefly sketch some possible approaches to deal with this assumption.

Section~2.2 of \cite{AdaptiveFEMBeltrami} argues that the lift introduces perturbations which are asymptotically negligible, thus meshes will not be distorted arbitrarily. Further, using particular refinement algorithms, like the newest-vertex bisection (NVB refinement), the number of elements in refined node-patches can be bounded using the number of elements per patch of the initial mesh. The error analysis of \cite{AdaptiveFEMBeltrami} does not assume shape-regularity, but the arising penalisation term is not included in their adaptive algorithm. This penalisation term was tracked during computation and was of higher order for their examples (see, e.g.,~\cite[Section~5.1]{AdaptiveFEMBeltrami}).

Section~1 of \cite{Bonito2013AFEMFG} proposes a strategy to guarantee the mesh intrinsic properties via some novel adaptive routine. For suitable initial meshes, a refinement process is constructed which consists of two parts: resolve the geometry, then refine based on the a posteriori error indicators. In \cite[Lemma~5.2]{Bonito2013AFEMFG} they have shown a condition, which guarantees that fine enough resolution of the initial surface is sufficient to retain shape-regularity. 

We will restrict the analysis to using newest-vertex bisection (NVB) and, following \cite{AdaptiveFEMBeltrami,APostElliptic,Bonito2013AFEMFG,BonitoDemlowNochetto}, to deriving appropriate a posteriori error estimates under the assumption that shape-regularity is retained, since reasonable refinement methods seem to maintain shape-regularity.

\textbf{Surface finite element semi-discretisation.} 
The discrete tangential gradient is given, see, e.g.,~\cite[Section~4.4]{Dziuk2013FiniteEM},  elementwise  by
\begin{equation}
\label{eq:discrete tangential gradient}
	\nabla_{\Ga_h} u_h := P_h \nabla \overline{u_h} := \nabla \overline{u_h} - (\nabla \overline{u_h} \cdot \nu_h) \nu_h  ,
\end{equation}
 where $\overline{u_h}$ is an extension of $u_h$ into $U$ which is constant along normals. The discrete tangential gradient is also independent of the extension.  

The finite element space on $\Ga_h$ is given as
\begin{equation}
\label{eq:def_S_h}
\begin{aligned}
S_h &= \{\phi \in C^0(\Ga_h) \mid \phi|_T  \, \, \text{linear affine for all} \, \, T \in \Ga_h \} \\
&= \text{span} \{\phi_1, \dots, \phi_{N}\} .
\end{aligned}
\end{equation}
The semi-discrete problem then reads: Find $u_h(\cdot,t) \in S_h$ such that
\begin{equation}
\label{eq:semi-discrete heat}
\left(\partial_t u_h(\cdot,t),v_h \right)_{\Ga_h} + (\nabla_{\Ga_h} u_h(\cdot,t), \nabla_{\Ga_h} v_h)_{\Ga_h} = (f_h(\cdot,t),v_h)_{\Ga_h} , \quad \text{for all } v_h \in S_h ,
\end{equation}
where $u_h(\cdot,0) = \widetilde{\pi}_h u^0$ and $f_h(\cdot,t) = \widetilde{\pi}_h f(\cdot,t)$ are, respectively, the $L^2$-projections of the initial data and right-hand side. 
We note that other interpolations, e.g., of Clément- or Scott--Zhang-type, are also viable, as is suggested in \cite{Verfuerth2003}; for their surface version, we refer to \cite[Chapter 2.4]{APostElliptic} and \cite[Chapter 3.1]{AdaptiveFEMBeltrami}.

\textbf{Lift.} 
To take the approximation back onto $\Ga$, points and functions are lifted between $\Ga_h$ and $\Ga$ via the closest point projection  within the strip $U$, see, e.g.,~\cite[Section~4.1]{Dziuk2013FiniteEM} or \cite[Section~2.1]{AdaptiveFEMBeltrami}:  The lifted point $x^\ell \in \Ga$ is the unique solution of  the closest-point projection  
\begin{equation}
\label{eq:closest point projection}
	x^\ell := x - d(x) \, \nu (x^\ell) \qquad \text{for $x \in \Ga_h \subset U$}.
\end{equation}
 The uniqueness of $x^\ell$ follows by the $C^2$-regularity of $\Ga$ and smallness of the width $\delta$ of the strip $U$. This further implies that the map from $\Ga_h$ to $\Ga$, given by the lift $x \mapsto x^\ell$, is bijective.  

Consequently, the \emph{lift} of a function $w \colon \Ga_h \rightarrow \R$ onto $\Ga$ is given by $w^{\ell} (x^\ell) := w (x)$, while the \textit{unlift}  $w^{-\ell} \colon \Ga_h \to \R$  is defined  such that  $(w^{-\ell})^\ell = w  \colon \Ga \to \R$ holds.  

Dziuk \cite[Lemma~3]{Dziuk1988} derived norm equivalences for lifted functions, for any $\eta_h \colon \Ga_h \to \R$,
\begin{equation}
\label{eq:norm_equiv}
	\begin{gathered}
		\frac{1}{c} \|\eta_h\|_{L^2(\Ga_h)} \leq \|\eta_h^\ell\|_{L^2 (\Ga)} \leq c \|\eta_h\|_{L^2(\Ga_h)} , \\
		\frac{1}{c} \|\nabla_{\Ga_h} \eta_h\|_{L^2(\Ga_h)} \leq \|\nabla_{\Ga} \eta_h^\ell\|_{L^2 (\Ga)} \leq c \|\nabla_{\Ga_h} \eta_h\|_{L^2(\Ga_h)} .
	\end{gathered}
\end{equation}

\subsection{Full discretization}
We introduce $K$ time steps $0=t^0 < t^1 < \dotsb < t^K = T_{\max}$ which build intervals $[t^{n-1}, t^n]$ of length $\tau^n$, such that $\sum_{j=1}^n \tau^j = t^n \leq T_{\max}$. For each time step we assume that an admissible triangulation $\Ga_h^n$ is given.

Even though \eqref{eq:heat_strong} is posed on a stationary surface, due to refinement and coarsening the discrete surface may change in each time step, i.e., $\Ga_h^{n-1} \neq \Ga_h^{n}$ in general. This temporal dependence will be reflected by the superscript $^n$, e.g., for discrete surfaces $\Ga_h^n$, finite element spaces $S_h^n:= \text{span} \{\phi_1^n, \dotsc, \phi_{N^n}^n\}$, see Section~\ref{section:SFEM}, with time-dependent degrees of freedom $N^n$, etc. We employ the same convention for the lift $^{\ell^n}: \Ga_h^n \to \Ga$.   
If the time interval is clear from the context we omit the additional superscript.

\textbf{The backward Euler method on variable meshes.}
To be able to define the implicit Euler time discretization on consecutive meshes $\Ga_h^{n-1}$ and $\Ga_h^n$ we need to introduce a \textit{refinement interpolation operator} $\IntRef \colon S_h^{n-1} \rightarrow S_h^n$. This interpolation operator maps functions on $\Ga_h^{n-1}$ to functions on $\Ga_h^n$.
Therefore, it is determined by the
refining and coarsening of subsequent surface meshes at time $t^{n-1}$ and $t^n$.  The refinement interpolation is described in detail below, see also Section~\ref{section:refinement interpolation - implementation}. 

This refinement interpolation enables us to define the backward difference time discretization:
\begin{equation*}
\label{discrete_time_deriv}
\partial_\tau  u_h (x,t) =  \frac{1}{\tau^n} \Big( u_h^n (x) - \IntRef u_h^{n-1} (x) \Big) \qquad n\geq 1, \ t \in (t^{n-1},t^n], \ x \in \Ga_h^n.
\end{equation*}

Thus, the full discretization of \eqref{eq:weak_form} reads: Find a sequence $(u_h^n)_{n=0}^K$ such that $u_h^n \in S_h^n$ solves, for $n \geq 1$,
\begin{equation}
\label{eq:discreteheat}
\left( \frac{u_h^n - \IntRef u_h^{n-1}}{\tau^n},v_h \right)_{\Ga_h^n} + (\nabla_{\Ga_h^n} u_h^n , \nabla_{\Ga_h^n} v_h)_{\Ga_h^n} = (f_h^n,v_h)_{\Ga_h^n} \qquad \forall v_h \in S_h^n,
\end{equation}
with initial data $u_h^0 = (\pi_h u^0)^{-\ell}$.

\textbf{Refinement interpolation.} 
\label{sec_sub:ref_int}
The refinement interpolation operator $\IntRef$ is based on the following idea: At a time step $t^{n}$ the refinement of the surface $\Ga_h^{n-1}$ yields an intermediate mesh with additional nodes. The surface $\Ga_h^n$ is then defined by lifting all these new nodes onto $\Ga$, to re-establish the interpolation property of the discrete surface.

 The values of the finite element function $u_h^{n-1}$ are computed in the new nodes before lifting them. The newly obtained values on the intermediate surface are combined with the pre-existing nodal values of $u_h^{n-1}$. Finally, the linear combination of the basis on $\Ga_h^n$ by the combined nodal values defines $\IntRef u_h^{n-1}$.  
Coarsening is handled analogously,  where the nodal values of removed nodes are simply forgotten.  
See Section~\ref{section:refinement interpolation - implementation} for  precise  details on the implementation of this process.

\textbf{Piecewise linear interpolation in time.} 
 
Given $u_h^n \in S_h^n$, for $n=0,\dotsc,K$, we define two piecewise linear interpolations: one on the \emph{continuous} and one on the \emph{discrete} surface.

We construct our approximation $u_{h, \tau}$ to $u$ by lifting the discrete solutions of \eqref{eq:discreteheat} and affine interpolation in time. We define the following function on $u_{h,\tau} \colon \Ga \times [0,T_{\max}] \to \R$:
\begin{equation}
\label{eq:lifted_discrete_sol}
	u_{h,\tau} (x,t) := \frac{t-t^{n-1}}{\tau^n} \big( u_h^n \big)^{\ell^n} (x) + \frac{t^n -t}{\tau^n} \big( u_h^{n-1} \big)^{\ell^{n-1}} (x) , \qquad \text{for any $t \in [t^{n-1},t^n]$}.
\end{equation}
Further we define a piecewise finite element function on the discrete surfaces \linebreak $\baruhtau(x,t) \colon \big(\Ga^0 \times \{0\}\big) \cup \big(\bigcup_{n=1}^K \Ga_h^n \times (t^{n-1},t^n]\big) \rightarrow \R$:
\begin{equation}
\label{eq:fully discrete solution definition}
	\baruhtau(x,t) = \frac{t-t^{n-1}}{\tau^n} u_h^n(x) + \frac{t^n -t}{\tau^n} \IntRef u_h^{n-1} (x) , \qquad \text{for any $t \in (t^{n-1}, t^n]$}.
\end{equation}
and $\baruhtau(x,0) = u_h ^0(x)$.
\begin{remark}
\label{remark:uhtau functions}
	(i) The unlift of $u_{h,\tau} \colon \Ga \to \R$ to the discrete surface $\Ga_h^n$ is not trivial, since the involved lift operation at time $t \in [t^{n-1},t^n]$ could arise from two different discrete domains $\Ga_h^{n-1}$ and $\Ga_h^n$.
	
	(ii) It is crucial to note that the function $u_{h,\tau}$ is $C^1$ in time, whereas, due to coarsening, its variant $\baruhtau$ on the discrete surface is \emph{discontinuous}, but piecewise differentiable. Therefore, time derivatives $\partial_t \baruhtau$ are understood piecewise on time intervals $(t^{n-1}, t^n]$. 
\end{remark}

\subsection{Smallest common triangulation and corresponding finite element space}
\label{sec:common_tri}

\paragraph{Smallest common refinement}
Following \cite[Section~2.2.2]{AdaptiveAlgHeatEq} we introduce the smallest common refinement of subsequent meshes $\Ga_h^{n-1}$ and $\Ga_h^n$, denoted by $\Ga_h^{\comesh} := \Ga_h^{n-1} \oplus \Ga_h^n$, which will serve as a common triangulation to compare functions on these meshes. Note that the discrete surfaces $\Ga_h^{n-1}$ and $\Ga_h^n$ may not be refinements of one another, but they both and $\Ga_h^{\comesh}$ are refinements of the initial mesh $\Ga_h^0$. It is worthwhile to note, that the main difficulty in constructing the common refinement is the requirement to lift all nodes of a mesh onto $\Ga$.

The smallest common refinement is defined by the main idea: $\Ga_h^{\comesh}$ is the collection of all common elements of the parent meshes, and always the refined part of $\Ga_h^n$ or $\Ga_h^{n-1}$ if an element was refined or coarsened.

Nodes which appear in both parent meshes are referred to as common nodes of $\Ga_h^{n-1}$ and $\Ga_h^n$ (marked by $\bullet$ in the top row of Figure~\ref{fig:nodal_values_of_common_triangulation}). 

The elements of the smallest common refinement $\Ga_h^{\comesh}$ are:
\begin{enumerate}
	\item[--] An identical element spanned by common nodes in both parent meshes $\Ga_h^{n-1}$ and $\Ga_h^n$ (see, e.g., the middle element in Figure~\ref{fig:nodal_values_of_common_triangulation}).
	\item[--] A set of elements of $\Ga_h^{n-1}$ which were coarsened to a single element of $\Ga_h^{n}$ spanned by common nodes (see, e.g., the left-hand side of $\Ga_h^{n-1}$ in Figure~\ref{fig:nodal_values_of_common_triangulation}).
	\item[--] A set of elements of $\Ga_h^n$ obtained by refinement of a single element of $\Ga_h^{n-1}$ spanned by common nodes (see, e.g., the right-hand side of $\Ga_h^n$ in Figure~\ref{fig:nodal_values_of_common_triangulation}).
\end{enumerate}
The smallest common refinement $\Ga_h^{\comesh}$ is sketched in the bottom row of Figure~\ref{fig:nodal_values_of_common_triangulation}.

Note that the interpolating points between common nodes, marked by grey {\tiny $\blacksquare$} in the top row of Figure~\ref{fig:nodal_values_of_common_triangulation}, are not nodes of the mesh $\Ga_h^{n-1}$ or $\Ga_h^n$. We denote these meshes by $\widehat{\Ga}_h^{n-1}$ and $\widehat{\Ga}_h^n$, respectively, which by construction satisfy: the mesh obtained by lifting the nodes of either $\widehat{\Ga}_h^{n-1}$, or $\widehat{\Ga}_h^{n}$ yields the discrete surface $\Ga_h^{\comesh}$.  In fact, only the interpolating points needed to be lifted, i.e.~the ones marked by grey {\tiny $\blacksquare$} in the top row of Figure~\ref{fig:nodal_values_of_common_triangulation}.

\begin{figure}[htbp]
	\centering
	\includegraphics[width=\linewidth]{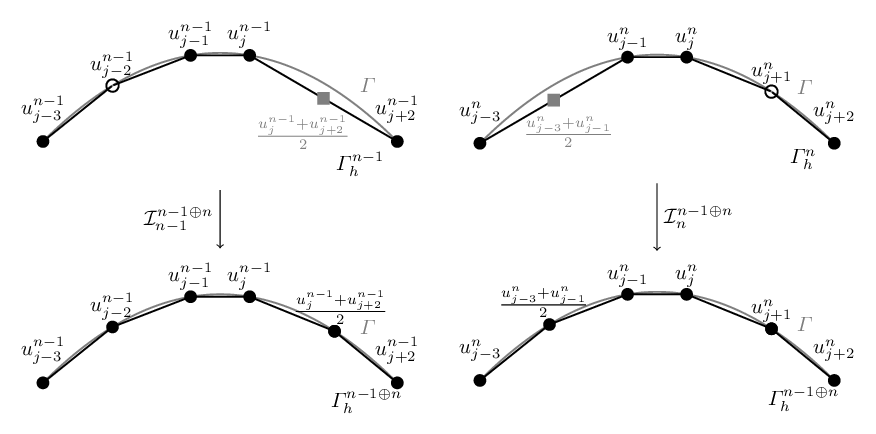}
\caption{Construction of smallest common refinement $\Ga_h^{\comesh}$ and common interpolation operators $\mathcal{I}_k^{\comesh}$ ($k=n-1,n$). The $\bullet$ marks (top row) denote common nodes of $\Ga_h^{n-1}$ and $\Ga_h^n$. The common triangulation (bottom row) keeps the more refined parts of the parent triangulations. (The nodal values are numbered locally here.)}
\label{fig:nodal_values_of_common_triangulation}
\end{figure}

\paragraph{Common finite element space}
The common finite element space $S_h^{\comesh}$ corresponding to the smallest common refinement $\Ga_h^{\comesh}$ is defined exactly as in \eqref{eq:def_S_h}.

In order to map functions on $\Ga_h^{n-1}$ or $\Ga_h^n$ to functions on $\Ga_h^{\comesh}$, we define the common interpolation operators $\mathcal{I}_k^{\comesh} \colon S_h^k \rightarrow S_h^{\comesh}$ for both $k = n-1$ and $k=n$, respectively.
 
The nodal values of $\mathcal{I}_k^{\comesh}$ are given by:
\begin{enumerate}
	\item[--] All nodes which appear in both $\Ga_h^{n-1}$ and $\Ga_h^{\comesh}$ or $\Ga_h^n$ and $\Ga_h^{\comesh}$, respectively, keep their nodal values (e.g., nodes marked by $\bullet$ and $\circ$ in the top row of Figure~\ref{fig:nodal_values_of_common_triangulation}).
	\item[--] Nodes which are not present in either $\Ga_h^{n-1}$ or $\Ga_h^n$ are assigned a nodal value using the Lagrangian interpolation (e.g., marked by grey {\tiny $\blacksquare$} in the top row of Figure~\ref{fig:nodal_values_of_common_triangulation}).
\end{enumerate}
\ebk

\section{A reliable and efficient residual-based error estimator}
\label{section:ReliablePara}

We use residual-based error analysis to derive a reliable and efficient error indicator (up to the data oscillation term \eqref{eq:osc} high-order geometric \eqref{eq:indicator - geo} and coarsening \eqref{eq:indicator - coarse} terms) which forms the basis of a space--time adaptive algorithm for parabolic surface PDEs.
 Where\-as it is usual to exclude the data oscillation term in the a posteriori analysis, see \cite{Verfuerth2003}.  The geometric and coarsening terms can also be excluded in the analysis for the following reasons:
The geometric terms and part of the coarsening term are of higher order and typically not dominant, but if needed, they can be handled separately in computation by the strategies of \cite{Bonito2013AFEMFG}. As the coarsening indicator \eqref{eq:indicator - coarse} can be made arbitrarily small by coarsening less elements, it also is not of concern in the a posteriori analysis.

\subsection{Residual}
\label{Sec_Res}

For any $t \in (t^{n-1}, t^n]$ and $v \in H^1 (\Ga)$, the residual to \eqref{eq:fully discrete solution definition} is defined by
\begin{equation}
\label{residual}
	\langle \mathcal{R}(u_{h,\tau}), v \rangle = \left(\partial_t u_{h,\tau} (t) , v \right)_{\Ga} + (\nabla_\Ga u_{h,\tau} (t), \nabla_\Ga v)_{\Ga} - (f(t),v)_{\Ga} .
\end{equation}
The error indicators are given by 
\begin{subequations}
\label{eq:indicator}
	\begin{align}
		\label{eq:indicator - split}
		\eta^n &\phantom{:}= (1+h) \Big( \tau^n \Big( (\eta_{h}^n)^2 +  (\eta_{\tau}^n)^2\Big) \Big)^{\frac{1}{2}} +  \Big( \tau^n\Big( (\mathcal{G}^n)^2 + (\eta_c^n)^2 \Big) \Big)^{\frac{1}{2}}   , \qquad \text{with} \\
		\label{eq:indicator - spatial}
		(\eta_{h}^n)^2 &:= \! \sum_{S \in \mathcal{S}_h^n} \!\! h_S \big\|\llbracket \nabla_{T} u_h^n \cdot \mathrm{n}_S \rrbracket \big\|_{L^2(S)}^2 \! +  \!\! \sum_{T \in  \mathcal{T}_h^n} \!\! h_T^2 \Big\|\frac{1}{\tau^n} (u_h^n - \IntRef u_h^{n-1})-f_h^n \Big\|_{L^2(T)}^2,  \\
		\label{eq:indicator - temporal}
		(\eta_{\tau}^n)^2 &:= \! \sum_{T \in \mathcal{T}_h^n} \big\|\nabla_{T} (u_h^n - \IntRef u_h^{n-1})\big\|_{L^2 (T)}^2, \\
		\label{eq:indicator - coarse}
		(\eta_{\textnormal{c}}^n)^2 &:= \! \sum_{T \in \mathcal{T}_h^{\comesh}} \frac{1}{(\tau^n)^2}\|\mathcal{I}_n^{\comesh} \IntRef u^{n-1}_h - \mathcal{I}_{n-1}^{\comesh} u^{n-1}_h\|_{L^2 (T)}^2 \\
		&\ \hspace*{4mm}+ \|\nabla_{\Ga_h^{\comesh}} \big(\mathcal{I}_n^{\comesh} \IntRef u^{n-1}_h - \mathcal{I}_{n-1}^{\comesh} u^{n-1}_h \big)\|_{L^2 (T)}^2 \nonumber\\
		 & \hspace*{-2em}+  \Big(h_T^2 + \frac{h_T^4}{(\tau^n)^2}\Big)  \Big( \| \nabla_{\Ga_h^{\comesh}} \mathcal{I}_{n-1}^{\comesh} u_h^{n-1} \|_{L^2 (T)}^2  + \|\nabla_{\Ga_h^{\comesh}} \mathcal{I}_n^{\comesh} \IntRef u^{n-1}_h \|_{L^2 (T)}^2 \Big), \nonumber\\
		\label{eq:indicator - geo}
		 (\mathcal{G}^n)^2 &:= \sum_{T \in \mathcal{T}_h^n}  h_T^4 \| \nabla_{\Ga_h^n} u_h^n \|_{L^2 (T)}^2  .
	\end{align}
\end{subequations}
where the set of all elements of $\Ga_h^n$ is denoted by $\mathcal{T}_h^n$, and the set of all edges is denoted by $\mathcal{S}_h^n$. Furthermore, the jump across an edge $S \in \mathcal{S}_h^n$ is given by $\llbracket w \rrbracket|_S := w|_{T_1} - w|_{T_2}$, where $T_1$ and $T_2$ are the two triangles from $\mathcal{T}_h^n$ sharing the edge $S$.  The tangential gradient on an element $T$ is denoted by $\nabla_T$,  and $\mathrm{n}_S$ is the outward edge-normal to $S$ with respect to $T_1$.

The global in space and local in time indicator $\eta^n$, is defined by the spatial and temporal indicators, $\eta_h^n$ and $\eta_\tau^n$, respectively. And further is dependent on the coarsening $\eta_{\textnormal{c}}^n$ and geometric errors $\mathcal{G}^n$. 

One can observe that the spatial and temporal indicators are related to indicators for the Euclidean case, see \cite[equations (4.8)--(4.10)]{Verfuerth2003} (handling oscillations slightly differently).  We state some crucial differences how the reliability and efficiency proofs for surface PDEs differ from their flat domain counterparts: 
\begin{itemize}
	\item[-] the $H^1$-norm required by the problem, 
	\item[-] the additional factor of $h$ of the spatial and temporal indicators in \eqref{eq:indicator - split}, 
	\item[-] the surface refinement interpolant $\IntRef$, 
	\item[-]  the high-order geometric terms $\mathcal{G}^n$, 
	\item[-]  the coarsening indicator $\eta_{\textnormal{c}}^n$ evaluated on a non-trivial common triangulation,
	\item[-] and the indicators being defined on a discrete surface. 
\end{itemize} 
These changes require a careful extension of the residual analysis known in the flat, Euclidean case.
\begin{remark}
Note that we choose to measure the residual with respect to $u_{h,\tau}$ due to the continuity, which allows us to show energy estimates needed for the equivalence of the error and the residual in Section~\ref{section:error and residual equivalence}.
\end{remark}
\subsection{Main result: reliability and efficiency}
We will now formulate the main result of this paper, which provides reliability and efficiency for the above defined residual-based error indicator. 

We will use the notation
\begin{equation}
	\label{eq:graph_norm}
	\|v\|_{X(s,t;\Ga)}^2 := \|v\|_{L^\infty (s,t; L^2 ( \Ga))}^2 +\|v\|_{L^2 (s,t; H^1 ( \Ga))}^2 + \|\partial_t v\|_{L^2 (s,t; H^{-1} ( \Ga))}^2 .
\end{equation}
% [kb: Something better as $\|\cdot\|_{X(s,t)}$?]  

\begin{theorem}
\label{thm:upper_lower}
	Let $h \leq h_0$, with a sufficiently small $h_0 > 0$, the residual-based error estimator $\eta^n$ of \eqref{eq:indicator}, and the error between the solution $u$ of \eqref{eq:weak_form} and  the numerical approximation $u_{h,\tau}$ \eqref{eq:lifted_discrete_sol},  obtained via \eqref{eq:discreteheat}, satisfies the following estimates, for $0 < t^n = \tau^1 + \dotsb + \tau^n \leq T_{\max}$:
	
	\begin{subequations}
	\indent(a) A global upper bound in space and time (reliability up to oscillation):
	\begin{equation}
		\label{eq:reliability estimate}
			\|u-u_{h,\tau} \|_{X(0,t^n;\Ga)} \leq c^\star \left(\sum_{j=1}^n (\eta^j)^2 + \|f -f_h^\ell \|_{L^2(0,t^n; H^{-1} (\Ga))}^2 + \|u^0-(u_h^0)^\ell \|_{L^2(\Ga)}^2 						\right)^{\frac{1}{2}} .
	\end{equation}
	\indent(b) A lower bound which is global in space and local in time (efficiency up to oscillation, geometric and coarsening defects):
	\begin{equation}
		\label{eq:efficiency estimate}			
			\eta^n \leq \ c_\star \Big( \|u-u_{h,\tau} \|_{X(t^{n-1},t^n; \Ga)} +\|f - f_h^\ell \|_{L^2(t^{n-1},t^n; H^{-1} (\Ga))} +  (\tau^n)^{\frac{1}{2}} \big(\mathcal{G}^n + \eta_c^n 			\big) \Big).
	\end{equation}
	\end{subequations}
	The constants $c_\star > 0$ and $c^\star > 0$ are independent of $h$, $t^n$, and $\tau^n$, but depend on the shape-regularity constant $\varrho^n$ of $\Ga_h^n$, and on $\Ga$. The constant $c^\star$ additionally depends on the shape-regularity constants $\varrho^j$ of the prior meshes $\Ga_h^j$.
\end{theorem}

Theorem~\ref{thm:upper_lower} will be proved in the subsequent section, which maintains the structure (marked by (a) and (b)) of the above theorem: results and estimates marked with (a) and (b) respectively indicate which of them will be used to prove \eqref{eq:reliability estimate} and \eqref{eq:efficiency estimate}.

\begin{remark}
\label{remark:h0 sufficiently small}
	In an adaptive setting the assumption $h \leq h_0$ might seem counter-intuitive, as it restricts coarsening. It is, however, inevitable to ensure that the closest point projection \eqref{eq:closest point projection} is unique. It is further required by all geometric approximations, see, e.g.,~\cite{Dziuk2013FiniteEM}. The constant $h_0$ solely depends on the geometry of $\Ga$, and it enforces that throughout the adaptivity one cannot coarsen beyond some suitable initial triangulation, where the lift is bijective.
\end{remark}

\begin{remark}
\label{remark:Geometric term and coarsening indicator}
The above theorem does not include efficiency for the high-order geometric term and coarsening indicator. As stated in \cite[Section~4.2]{APostElliptic} the geometric term arising from the stiffness term of \eqref{eq:res_space} is not the main concern when dealing with convergence and optimality of an adaptive algorithm. As \cite[Lemma~5.8 and Chapter~6.1]{Bonito2013AFEMFG} suggests they can be handled utilising an additional adaptive routine to guarantee that the geometric errors are bounded by the spatial indicator, which infers that the lower bound \eqref{eq:efficiency estimate} holds up to oscillation. The terms introduced by the coarsening, are of a similar structure or can be made arbitrarily small by coarsening less. 
\end{remark}

\section{Proof of Theorem~\ref{thm:upper_lower}}
\label{section:proofs}

The proof can be summarised as follows:

We start by showing the equivalence of the fully discrete residual $\mathcal{R}(u_{h,\tau})$ and the error $u-u_{h,\tau}$ in Section~\ref{section:error and residual equivalence}.

This enables us to decompose the residual $\mathcal{R}(u_{h,\tau})$ of \eqref{residual} into spatial, temporal, coarsening, and geometric residuals, and a data oscillation term (Section~\ref{sec:decomposition}). Then we focus on each residual separately.

An integral transformation identity and geometric bound (Section~\ref{section:integral trafo and geom estimates}) will be derived to bound the geometric residual by the spatial, temporal,  and geometric  indicators (Section~\ref{sec:geo_h1_res}). 
Afterwards the substantially new coarsening residual is analysed in Section~\ref{sec:res_coarse} where we derive tools to extend ideas for the flat case to surfaces. Our bounds yield further geometric error contributions and yield a framework to compute coarsening indicators which we expect to be applicable to moving domains.
The spatial and temporal residuals are carefully constructed such that the main ideas for their analysis in the flat case can be extended to surfaces (Section~\ref{section:spatial residual bound}--\ref{section:temporal residual bound}).

Finally, the combination of these results completes the proof of the theorem (Section~\ref{section:thm proof}).

\medskip
By $c$ we will always denote a positive constant, that is independent of $h$, $\tau^n$, and $n$ but may change its value between steps. 
Many results hold for a mesh size $h \leq h_0$, which is always understood with a  sufficiently small $h_0 > 0$, see Remark~\ref{remark:h0 sufficiently small}.

\subsection{Equivalence of the error and the residual}
\label{section:error and residual equivalence}

We now show that the residual and the error between $u$ and $u_{h,\tau}$, see \eqref{eq:lifted_discrete_sol}, are equivalent. Although the following equivalence result is formally the same as the one for the parabolic flat case \cite[Lemma~4.1]{Verfuerth2003}, its energy estimate-based proof requires additional steps due to the full $H^1$ setting of our problem.
\begin{proposition}
\label{prop:error and residual equivalence}
	For any $t^n \in (0,T]$, the error and the residual satisfy the following estimates in the graph norm \eqref{eq:graph_norm}:
	\begin{subequations}
		\begin{align}
			\label{eq:res_equiv_2}
			\|u -  u_{h,\tau} \|_{X(0,t^n;\Ga)} 
			\leq &\ c \left( \|u^0 - (u_h^0)^\ell \|^2 _{L^2(\Ga)} + \|\mathcal{R}(u_{h,\tau}) \|^2 _{L^2(0,t^n; H^{-1} (\Ga))} \right)^{1/2} , \\
			\label{eq:res_equiv_1}
			\int_0^{t^n} \langle \mathcal{R}(u_{h,\tau}), w \rangle \d t 
			\leq &\ \|u - u_{h,\tau} \|_{X(0,t^n;\Ga)} \|w \|_{L^2(0,t^n; H^1 (\Ga))} ,
		\end{align}
	\end{subequations}
	where the constant $c>0$ depends on $T$, but is independent of $h$ and $\tau^n$.
\end{proposition}

\begin{proof}
	The proof relies on the error equation, which -- recalling \eqref{eq:lifted_discrete_sol} and Remark~\ref{remark:uhtau functions} -- reads, for all $t \in [0,T]$ and for any $v \in H^1(\Ga)$:
	\begin{equation}
		\label{eq:proof_equiv_err_res}
		(\partial_t (u-  u_{h,\tau}),v)_{\Ga} + (\nabla_\Ga (u-  u_{h,\tau}),\nabla_\Ga v)_{\Ga} 
		%	= (f,v)_{\Ga} - (\partial_t  u_{h,\tau},v)_{\Ga} -(\nabla_\Ga  u_{h,\tau},\nabla_\Ga v)_{\Ga} 
		= \langle \mathcal{R}(u_{h,\tau}),v \rangle,
	\end{equation}
	which is obtained by subtracting the definition of the residual \eqref{residual} from the weak formulation \eqref{eq:weak_form}.  
	
	(a) The upper bound \eqref{eq:res_equiv_2} is shown using an energy estimate. For $t \in (t^{n-1},t^n]$, we test the error equation \eqref{eq:proof_equiv_err_res} by $e := u -  u_{h,\tau}$, which yields (dropping the omnipresent $\Ga$ within norms)
	\begin{align*}
		\frac{1}{2} \frac{\d}{\d t} \|e\|_{L^2}^2  + \|\nabla_\Ga e\|_{L^2}^2 = \langle \mathcal{R}(u_{h,\tau}),e \rangle &\leq \|\mathcal{R}(u_{h,\tau})\|_{{H^{-1}}} \|e\|_{H^1} \\
		 &\leq \|\mathcal{R}(u_{h,\tau})\|_{{H^{-1}}} (\|e\|_{L^2} + \|\nabla_\Ga e\|_{L^2}).
		%	&\quad \Rightarrow \frac{\d}{\d t} \|e\|_{L^2}^2 + \|\nabla e\|_{L^2}^2 \leq \|\mathcal{R}(u_{h,\tau})\|_{{H^{-1} (\Ga)}}^2 + \|e\|_{L^2}. 
	\end{align*}
	Upon applying Young's inequality and an absorption of the $H^1$-seminorm to the right-hand side, integrating the inequality over $[0,t^n]$,
	and using Gronwall's inequality yields
	\begin{equation*}
		\label{eq:proof_res_equiv_A}
		\|e\|_{L^\infty (0,t^n;L^2)}^2 \leq \exp({t^n}) \big( \|\mathcal{R}(u_{h,\tau})\|_{L^2(0,t^n;H^{-1})}^2 + \|e(0)\|_{L^2}^2 \big) .
	\end{equation*}
	which directly infers the $L^\infty(L^2)$ bound and the $L^2(H^1)$ estimate, via $\|e\|_{L^2(0,t^n;L^2)} \leq t^n \|e\|_{L^\infty(0,t^n;L^2)}^2$, i.e.\
	\begin{equation*}
		\|e\|_{L^2 (0,t^n;H^1)}^2  \leq (t^n \exp({t^n}) +1) \big( \|\mathcal{R}(u_{h,\tau})\|_{L^2(0,t^n;H^{-1})}^2 + \|e(0)\|_{L^2}^2 \big) .
	\end{equation*}
	
	To show the $L^2(H^{-1})$ estimate for $\partial_t e$, we rearrange \eqref{eq:proof_equiv_err_res}, and write
	\begin{equation*}
		\| \partial_t e \|_{H^{-1}} 
		= \sup_{\substack{v \in H^1(\Ga) \\ v \neq 0}} \frac{(\partial_t e, v)}{\|v \|_{H^1}}
		= \sup_{\substack{v \in H^1(\Ga) \\ v \neq 0}} \frac{ \langle \mathcal{R}(u_{h,\tau}),v \rangle - (\nabla_\Ga e,\nabla_\Ga v)_{\Ga} }{\|v \|_{H^1}} ,
	\end{equation*}
	which yields (in an interval-wise sense)
	\begin{equation*}
		\label{eq:proof_res_equiv_C}
		\| \partial_t e \|_{L^2(0,t^n;H^{-1})}^2 
		\leq 
		%	2 \|\mathcal{R}(u_{h,\tau})\|_{L^2(0,t^n;H^{-1})}^2 + 2 \int_{0}^{t^n} \|\nabla e(s)\|_{L^2}^2 \d s =
		2 \|\mathcal{R}(u_{h,\tau})\|_{L^2(0,t^n;H^{-1})}^2 + 2 \|\nabla_\Ga e\|_{L^2 (0,t^n;L^2)}^2.
	\end{equation*}
	
	This completes the proof of the first estimate.
	
	(b) The lower bound \eqref{eq:res_equiv_1} directly follows by recalling \eqref{eq:proof_equiv_err_res} and \eqref{eq:graph_norm}, integrating over time, and applying standard estimates.
\end{proof}

\subsection{Residual decomposition}
\label{sec:decomposition}

 The equivalence result of Proposition~\ref{prop:error and residual equivalence} uses the approximation ${u_{h,\tau}}$ on the continuous surface $\Ga$, defined in \eqref{eq:lifted_discrete_sol}. This complicates important arguments in the a posteriori error analysis, since we are not able to simply unlift this function -- see Remark~\ref{remark:uhtau functions} -- and argue via norm equivalence. 

In order to guarantee the computability of the error indicators we insert the approximation $\baruhtau$ on the discrete surface, \eqref{eq:fully discrete solution definition}, into the residual equation \eqref{residual} as a zero. This will require us to deal with the difference due to using $\baruhtau$ instead of $u_{h,\tau}$, and will naturally introduce a coarsening indicator (similarly to \cite[Section~4.4]{Bartels2016}, \cite{chen2004feng_adap_flat, AdaptiveAlgHeatEq}). In spirit, this approach is from the analysis in the flat parabolic case,  (see, e.g.,~\cite{AdaptiveAlgHeatEq}, \cite[Section~4.4.2]{Bartels2016}), representing finite element functions on a joint mesh.

We divide the residual \eqref{residual} into a spatial, a temporal, a coarsening, a geometric, and an oscillation term, that is we respectively set:
\begin{align}
\label{eq:res_eq}
	\mathcal{R}({u_{h,\tau}}) = \mathcal{R}_h + \mathcal{R}_\tau  + \Rcoarse + \mathcal{R}_{\textnormal{g}} + \osc (f).
\end{align}

To state the decomposition we introduce the following notations: 
The Jacobian of the closest point projection $\mu_h^n := \frac{\d \sigma}{\d \sigma_h^n}$ expressed by the surface measures $\d \sigma$, $\d \sigma_h^n$ on $\Ga$ and $\Ga_h^n$ respectively. Note that $\mu_h^n$ can expressed in terms of the curvature, distance function, discrete and continuous normal vectors, see \cite[Proposition~2.1]{AdaptiveFEMBeltrami}. 
The projection $P_h^n$ from \eqref{eq:discrete tangential gradient}, and the operator 
\begin{equation}
\label{eq:projection operator R_h}
	\tilde{R}_h^n := \mu_h^n P_h^n \Big(I-\frac{\nu_h^n \nu^T}{\nu_h^n \cdot \nu}\Big)^T (I-d \mathcal{A})^{-1} (I-d \mathcal{A})^{-1} \Big(I-\frac{\nu_h^n \nu^T}{\nu_h^n \cdot \nu}\Big) ,
\end{equation} 
where $\mathcal{A}$ is the extended Weingarten map, and $d$ the signed distance function, see Section~\ref{section:surface_operators} and \ref{section:SFEM}.  Recall the definition of the unlift operation for functions, see Section~\ref{section:SFEM}. 
The operator $\tilde{R}_h^n$ is directly related to the one introduced in \cite[equation~(52)]{BonitoDemlowNochetto}.  

We define each term in \eqref{eq:res_eq} piecewise on each time interval as follows, for $n = 1, \dots, K$ let $t \in (t^{n-1}, t^n]$ and $v \in H^1(\Ga)$ arbitrary:

The \textit{spatial} residual $\mathcal{R}_h\equiv\mathcal{R}_h(t)$ (within $(t^{n-1}, t^n]$) is given by
\begin{subequations}
\label{eq:decom_subs}
\begin{align}
	\label{eq:res_space} 
	\langle \mathcal{R}_h, v \rangle = &\ \big( f_h^n, v^{-\ell} \big)_{\Ga_h^n} - \bigg( \frac{u_h^n - \IntRef u_h^{n-1}}{\tau^n}, v^{-\ell} \bigg)_{\Ga_h^n} \!\!\! - \big( \nabla_{\Ga_h^n} u_h^n , \nabla_{\Ga_h^n}  v^{-\ell} \big)_{\Ga_h^n}, \\
	\intertext{the \textit{temporal} residual $\mathcal{R}_\tau$ is given by} 
	\label{eq:res_temp} 
	\langle \mathcal{R}_{\tau} (t), v \rangle = &\ \big( \nabla_{\Ga_h^n} (u_h^n -\baruhtau),\nabla_{\Ga_h^n} v^{-\ell} \big)_{\Ga_h^n}, \\
	\intertext{the \textit{coarsening} residual $\Rcoarse$ is given by} 
	\label{eq:res_coarse} 
	\langle \Rcoarse (t), v \rangle = &\ \bigg( \partial_t \big((\baruhtau)^{\ell} - u_{h,\tau}\big),v \bigg)_{\Ga} \!\!\! + \Big( \nabla_\Ga \big((\baruhtau)^{\ell} - u_{h,\tau}\big),\nabla_\Ga  v \Big)_{\Ga} , \\
	\intertext{the \textit{geometric} residual $\mathcal{R}_{\textnormal{g}}$ is given by}
	\label{eq:res_geo} 
	\langle \mathcal{R}_{\textnormal{g}} (t), v \rangle = &\  \! \big( (1- \!\mu_h^n)(\partial_t  \baruhtau \! - \! f_h^n),v^{-\ell} \big)_{\Ga_h^n} \! + \! \big( (I \! - \tilde{R}_h^n)P_h^n \nabla_{\Ga_h^n} u_h^n, \nabla_{\Ga_h^n} v^{-\ell} \big)_{\Ga_h^n} \nonumber  \\
	&\ + \big( (I-\tilde{R}_h^n) P_h^n \nabla_{\Ga_h^n} ( \baruhtau-u_h^n ), \nabla_{\Ga_h^n} v^{-\ell} \big)_{\Ga_h^n}, \\
	\intertext{and the  (naturally $\Ga_h^n$-dependent)  data \textit{oscillation} by}
 	\label{eq:osc}	
	\osc(f(\cdot,\! t)) \! = &\ f(\cdot,t) - (f_h^n)^\ell.
\end{align}
\end{subequations}
 Recall, the notational convention regarding the lift, and that time derivatives of $\baruhtau$ are understood piecewise, see Remark~\ref{remark:uhtau functions}.

In contrast to all other residuals the coarsening residual is still based on integrals on the curved domain $\Ga$, we will introduce further tools in Section~\ref{sec:res_coarse} to be able to obtain the easily computable coarsening indicator \eqref{eq:indicator - coarse}, which is (up to higher order geometric contributions) similar to the coarsening indicators in flat domains \cite{chen2004feng_adap_flat, AdaptiveAlgHeatEq}.

\subsection{An integral transformation and a geometric estimate}
\label{section:integral trafo and geom estimates}

 To obtain the residual decomposition we insert $\pm (\baruhtau)^\ell$, the lifted version of \eqref{eq:fully discrete solution definition} into \eqref{residual}, the difference with $u_{h,\tau}$ yields the coarsening residual $\Rcoarse$.
By the linearity of the residual we are left with $\mathcal{R}((\baruhtau)^\ell)$ which simplifies the bounds of all other residuals.  We use the known integral transformations from $\Ga_h$ to $\Ga$, see, e.g.,~\cite[Lemma~5.1--5.2]{ESFEM2007}, or \cite[Lemma~4.7]{Dziuk2013FiniteEM}. We additionally use a transformation for integrals from $\Ga$ to $\Ga_h$,   see also \cite[Lemma~21 and Corollary~33]{BonitoDemlowNochetto}  . We note that this transformation is similar, yet substantially different, to the identity \cite[equation~(2.2.22)]{AdaptiveFEMBeltrami}, transforming integrals from $\Ga_h$ to $\Ga$. In particular the transformation from \cite{AdaptiveFEMBeltrami} is not invertible.

Since all results in this section hold for discrete surfaces at \textit{any time} $t^n$, we have dropped all superscripts $^n$. 

\begin{lemma}
\label{lemma:new integral trafo}
	Let $\Ga_h$ be an admissible triangulation for $h \leq h_0$, and let $\tilde{A_h} = \tilde{R}_h P_h$, then the following identity holds
	\begin{equation*}
	\label{eq:integral_trafo_new}
		\int_\Ga \nabla_\Ga v \cdot \nabla_\Ga w = \int_{\Ga_h} \tilde{A}_h \nabla_{\Ga_h} v^{-\ell} \cdot \nabla_{\Ga_h} w^{-\ell}.
	\end{equation*}
\end{lemma}

\begin{proof}
The proof follows a similar approach as for the integral transformation in \cite[equation~(2.2.22)]{AdaptiveFEMBeltrami}. Use the representation of a tangential gradient by a discrete tangential gradient from \cite[equation~(2.2.19)]{AdaptiveFEMBeltrami}:%, for any $y \in \Ga$:
\begin{align}
\label{eq:grad_lifted_to_discrete}
	%\nabla_\Ga v (y) = (I - d \mathcal{A}]^{-1} \left(I - \frac{\nu_h \nu^T}{\nu_h \cdot \nu} \right) \nabla_{\Ga_h} v^{-\ell} (x(y)) := [I - d \mathcal{A}]^{-1} \ProjOneSided \nabla_{\Ga_h} v^{-\ell} (x(y)).
	\nabla_\Ga v = (I - d \mathcal{A})^{-1} \ProjOneSided \nabla_{\Ga_h} v^{-\ell} := (I - d \mathcal{A})^{-1} \left(I - \frac{\nu_h \nu^T}{\nu_h \cdot \nu} \right) \nabla_{\Ga_h} v^{-\ell}  .
\end{align}
In this formula, rewrite the scalar product of the tangential gradients (note that $P_h \nabla_{\Ga_h} = \nabla_{\Ga_h}$), to obtain
\begin{align*}
	\nabla_\Ga v \cdot \nabla_\Ga w 
%	&= (I - d \mathcal{A})^{-1} \ProjOneSided P_h \nabla_{\Ga_h} v^{-\ell} \cdot (I - d \mathcal{A})^{-1} \ProjOneSided P_h \nabla_{\Ga_h} w^{-\ell} \\
	&= P_h \ProjOneSided^T (I - d \mathcal{A})^{-1}  (I - d \mathcal{A})^{-1}  \ProjOneSided P_h \nabla_{\Ga_h} v^{-\ell}  \cdot \nabla_{\Ga_h} w^{-\ell}  \\
	&= \frac{1}{\mu_h} \tilde{A_h} \nabla_{\Ga_h} v^{-\ell} \cdot \nabla_{\Ga_h} w^{-\ell}.
\end{align*}
Combining this with the integral transformation $\int_\Ga v = \int_{\Ga_h} \mu_h v^{-\ell}$, see \cite[Lemma~5.1--5.2]{ESFEM2007}, finishes the proof.
\end{proof}

With this result, we see that the residual decomposition \eqref{eq:res_eq}--\eqref{eq:decom_subs} for \eqref{residual} indeed holds. 
The three terms of \eqref{residual} are, respectively, the first term of the oscillation \eqref{eq:osc}, the transformed term $\mu_h \partial_t \baruhtau$, and the term $\tilde{A}_h \nabla_{\Ga_h^n}  \baruhtau$ of the geometric residual \eqref{eq:res_geo}. 
The remaining terms vanish by rearranging and using the integral transformation.

To bound the geometric residual in Section~\ref{sec:geo_h1_res} we show an $L^\infty$-estimate for $P_h - \tilde A_h$.
\begin{lemma}
	\label{lem:trafo}
	Let $\Ga_h$ be an admissible triangulation for $h \leq h_0$.
	Then the following geometric estimate holds, with an $h$ independent constant $c > 0$,
	\begin{align*}
		\|P_h - \tilde{A}_h\|_{L^\infty (\Ga_h)} \leq c h^2 .
	\end{align*} 
\end{lemma}
\begin{proof}
	The surface has a bounded second fundamental form $\mathcal{A}$, and $\Ga_h$ satisfies $\|d\|_{L^\infty(\Ga_h)} \leq c h^2$, therefore, for $h \leq h_0$, we have $(I-d \mathcal{A})^{-1} = I + \mathcal{O}(h^2)$, via its Neumann series.  Upon recalling $\tilde{A}_h = \tilde{R}_h P_h$ (see \eqref{eq:discrete tangential gradient} and \eqref{eq:projection operator R_h}), we obtain
	\begin{align*}
		P_h - \tilde{A}_h &= P_h-\tilde{R}_h P_h = P_h-P_h \ProjOneSided^T \ProjOneSided P_h + \mathcal{O}(h^2) \\
		&= -(\nu- (\nu \cdot \nu_h) \nu_h)(\nu- (\nu \cdot \nu_h) \nu_h)^T +\mathcal{O} (h^2) = \mathcal{O} (h^2) ,
	\end{align*}
	where the final step is shown using the arguments of the proof of Lemma~4.1 in \cite{Dziuk2013FiniteEM}. 
\end{proof}

\subsection{Reducing the geometric residual}
\label{sec:geo_h1_res}

 We start by relating the geometric residual \eqref{eq:res_geo}, to the indicators \eqref{eq:indicator}.  
 
We use the classical geometric approximation results (all requiring $\Ga_h^n$ to be an admissible triangulation for $h \leq h_0$) for all $0 < t^n \leq T_{\max}$, $\|1-\mu_h^n\|_{L^\infty(\Ga_h^n)} \leq c h^2$ and $\|d\|_{L^\infty (\Ga_h^n)} \leq c h^2$, originally shown in \cite[Section~5]{Dziuk1988}, and Lemma~\ref{lem:trafo} from above.
\begin{lemma}
\label{lemma:geo}
	Let $\Ga_h^n$ be an admissible triangulation for $h \leq h_0$ for all $n$ such that $0 < t^n \leq T_{\max}$.
	Then there is an $h$- and $n$-uniform constant $c > 0$, depending only on $\Ga$, such that	
	\addtocounter{equation}{+1}
	\begin{equation} \tag{\theequation a}
		\label{eq:geo_res_bound}
		 \|\mathcal{R}_{\textnormal{g}}(t)\|_{{H^{-1} (\Ga)}} \leq c\Big(\mathcal{G}^n + h \big(\eta_h^n + \eta_\tau^n \big)\Big) \qquad \text{for all $t \in [t^{n-1},t^{n}]$}.
	\end{equation}
\end{lemma}
\begin{proof}
 By the definition of the geometric residual \eqref{eq:res_geo}, together with the estimates from \cite[Lemma~5.1]{ESFEM2007} and Lemma~\ref{lem:trafo}, for any $t \in [0,T_{\max}]$ belonging to $(t^{n-1}, t^n]$, we obtain
\begin{align*}
	\|\mathcal{R}_{\textnormal{g}}(t)\|_{H^{-1} (\Ga)} 
	&= \sup_{\substack{v \in H^1(\Ga) \\ v \neq 0}} \frac{|\langle \mathcal{R}_{\textnormal{g}}(t), v\rangle|}{\|v\|_{H^1(\Ga)}} \\
	&\hspace{-8mm}\leq ch^2 \Big(\|\partial_t   \baruhtau- f_h^n \|_{L^2(\Ga_h^n)} +\|\nabla_{\Ga_h^n}(\baruhtau-u_h^n)\|_{L^2(\Ga_h^n)} +\| \nabla_{\Ga_h^n} u_h^n \|_{L^2(\Ga_h^n)}   \Big) .
\end{align*}
Comparing to the error indicators \eqref{eq:indicator} we obtain the estimates:
\begin{align*}
\|\mathcal{R}_{\textnormal{g}}(t)\|_{H^{-1} (\Ga)}^2 \leq ch^2((\eta_h^n)^2 + h^2 (\eta_\tau^n)^2) + (\mathcal{G}^n)^2 &\overset{h \leq h_0}{\leq} ch^2 ((\eta_h^n)^2 +(\eta_\tau^n)^2) +(\mathcal{G}^n)^2 ,
\end{align*}
where we have estimated elementwise.
\end{proof} 

\subsection{Coarsening indicator on two distinct meshes}
\label{sec:res_coarse}

The coarsening residual is more complex than in the flat case. We derive an indicator which only contains discrete quantities, since evaluation of lifted functions is computationally too expensive. 
The coarsening residual measures the error occurring from lifting $u_h^{n-1}$ on two different meshes $(\Ga_h^{n-1})^{\ell^{n-1}}$ and on $(\Ga_h^{n})^{\ell^{n}}$ via the refinement interpolant:
\begin{equation*}
	(u_h^{n-1})^{\ell^{n-1}} - (\IntRef u^{n-1}_h)^{\ell^{n}}
\end{equation*} 

Recall the notions of the common triangulation and interpolation from Section~\ref{sec:common_tri}.

The benefit in using the common triangulation is that the difference of $(u_h^{n-1})^{\ell^{n-1}}$ and $(\IntRef u_h^{n-1})^{\ell^{n}}$ can be represented on $(\Ga_h^{\comesh})^\ell$ and then unlifted to the discrete common triangulation $\Ga_h^{\comesh}$. We will see that interpolation of a function on the common triangulation only introduces errors of higher order. We emphasize that these high-order terms are purely related to the geometry $\Ga$, and would be zero in the flat case.

We show the following upper bound.
\begin{lemma}
	 Let $\Ga$ be a surface such that its principal curvatures and their derivatives are bounded.  Let $\Ga_h^n$ be an admissible triangulation, obtained using NVB refinement, for $h \leq h_0$ for all $n$ such that $0 < t^n \leq T_{\max}$.
	Then there is an $h$- and $n$-uniform constant $c > 0$, depending only on $\Ga$, such that	
	\addtocounter{equation}{+1}
	\begin{equation} \tag{\theequation a}
	\label{eq:coarse_bound_new}
		\|\Rcoarse(t)\|_{H^{-1} (\Ga)} \leq c \, \eta_c^n .
	\end{equation}
\end{lemma}

\begin{proof}
The term of interest arises from the Cauchy--Schwarz inequality of \eqref{eq:res_coarse}:
\begin{equation}
	\label{eq:ind_coarse_not_simple}
	 \frac{1}{\tau_n} \big\| (\IntRef u^{n-1}_h)^{\ell^{n}}- (u^{n-1}_h)^{\ell^{n-1}} \big\|_{L^2(\Ga)} + \Big(\frac{t^n-t}{\tau^n} \Big)  \Big\| \nabla_\Ga \Big((u^{n-1}_h)^{\ell^{n-1}} -(\IntRef u^{n-1}_h)^{\ell^{n}} \Big)  \Big\|_{L^2(\Ga)}.
\end{equation}

When estimating the term $\|  (\IntRef u^{n-1}_h)^{\ell^{n}}- (u^{n-1}_h)^{\ell^{n-1}}  \|$ we face two main issues: the two functions are defined on different finite element nodes, and involve two distinct lifts. 

We alleviate these issues by using the common interpolation operators $\mathcal{I}_{n-1}^{\comesh}$, $\mathcal{I}_{n}^{\comesh}$, and the triangle inequality, which yield: 
\begin{align*}
\| (\IntRef u^{n-1}_h)^{\ell^{n}}- (u^{n-1}_h)^{\ell^{n-1}} \| \leq &\ \|(\mathcal{I}_n^{\comesh} \IntRef u^{n-1}_h)^{\ell^{\comesh}} - (\mathcal{I}_{n-1}^{\comesh} u^{n-1}_h)^{\ell^{\comesh}} \| \\
&\ + \| (\mathcal{I}_{n}^{\comesh} \IntRef u^{n-1}_h)^{\ell^{\comesh}} - (\IntRef u^{n-1}_h)^{\ell^{n}} \| \\
&\ + \| (\mathcal{I}_{n-1}^{\comesh} u^{n-1}_h)^{\ell^{\comesh}} - (u^{n-1}_h)^{\ell^{n-1}} \|.
\end{align*}

The first term here is directly appearing in the coarsening indicator \eqref{eq:indicator - coarse}, up to a norm equivalence in $S_h^{\comesh}$.

The second and third terms deal with interpolation errors which still involve two distinct discrete domains and lifts, but the underlying finite element functions have identical values in corresponding points, see Figure~\ref{fig:nodal_values_of_common_triangulation}. Next we deal with the issue of different discrete domains. 

We will present our argument for the third term, setting $w_h^0 = u_h^{n-1}$. The estimate for the second term is analogous, with $w_h^0 = \IntRef u^{n-1}_h$ in place of $u^{n-1}_h$.
We will first assume that at a time step an element is only refined once, using newest-vertex bisection, and comment on multiple refinements thereafter.

\subparagraph{Affine transformation}
Using the common triangulation $\Ga_h^{\comesh}$ (see Section~\ref{sec:common_tri}, and Figure~\ref{fig:nodal_values_of_common_triangulation}) as a basis, we use a $\theta$-argument very similar to the one developed in Section~4 of \cite{KLLP_2017}. 

To this end we introduce a continuous affine map which transforms a function from the discrete mesh $\Ga_h^0 := \widehat{\Ga}_h^{n-1}$ (with unlifted refined nodes $(x_j^0)$, see Figure~\ref{fig:nodal_values_of_common_triangulation} top row), to its common interpolated version  $\Ga_h^1 := \Ga_h^{\comesh}$ with nodes satisfying $x_j^1 = (x_j^0)^\ell$ by construction of the common triangulation (see Figure~\ref{fig:nodal_values_of_common_triangulation} bottom row). 
Recall that $x_j^1$ is the unique solution of the lift \eqref{eq:closest point projection}. 
The affine transformation is based on intermediate nodal values $x_j^\theta := x_j^0 + \theta (x_j^1 - x_j^0)$, for all $j$, which define a $\theta$-dependent discrete triangulation $\Ga_h^\theta$ and finite element spaces $S_h^\theta$.

On the intermediate surfaces, given a nodal vector $(w_i)$ we set $w_h^\theta \colon S_h^0 \times [0,1] \rightarrow S_h^\theta$, $w_h^\theta = \sum_i w_i \phi_i^\theta$ with fixed nodal values $w_i$ and the nodal basis functions $\phi_i ^\theta \in S_h^\theta$. Observe that elementwise $\phi^\theta (F(x,\theta)) = \phi^0(x)$, where $F \colon T^0 \times \theta \rightarrow T^\theta$ is an affine transformation (as the reference map in standard finite element analysis) which is linearly dependent on $\theta$. 
Further we introduce the $\theta$-dependent lift $\ell^\theta$, mapping $\Ga_h^\theta \to \Ga$. 
This construction is sketched in Figure~\ref{fig:theta construction}. 

Then for any point $a$ on the curved domain $\Ga$ we rewrite the difference of two lifted functions as follows (cf.~proof of Lemma~4.1 in \cite{KLLP_2017}):
\begin{equation}
	\begin{aligned}
		(w_h^1)^{\ell^1}(a)-(w_h^0)^{\ell^0}(a) 
		= &\ \int_0^1 \frac{\d}{\d \theta}(w_h^\theta)^{\ell^\theta}(a) \d \theta 
		= \int_0^1 \frac{\d}{\d \theta} \Big(w_h^\theta (x^\theta(a)) \Big) \d \theta \\
 = &\  \int_0^1 (\nabla_{\Ga_h^\theta} w_h^\theta)(x^\theta(a)) \cdot \partial_\theta x^\theta(a) + \partial_\theta w_h^\theta(x)|_{x =x^\theta(a)} \d \theta  \\
		= &\  \int_0^1 \big(\nabla_{\Ga_h^\theta} w_h^\theta\big)^{\ell^\theta}\cdot \big(\partial_\theta x^\theta(a) - A^\prime(\theta) A^{-1}(\theta)(x^\theta(a)-b) \big) \d \theta. \\
	\end{aligned}
\end{equation} 
We employ a transport property argument to replace the $\theta$-derivative of $w_h^\theta$. Since the partial derivative assumes a fixed $x \in T$, by the standard reference element map we have a fixed reference point $\hat{x}$ in the reference simplex $\hat{T}$. Movement under the $\theta$-dependent map $F(\hat{x},\theta)$ fulfils the transport property \cite[Proposition~5.4]{ESFEM2007}, i.e.\
\begin{align*}
0 = \frac{\d}{\d \theta} w_h^\theta(F(\hat{x},\theta)) = \big[D_x w_h^\theta(x) \cdot \partial_\theta F(\hat{x},\theta) + \partial_\theta w_h^\theta(x)\Big]_{x = F(\hat{x},\theta)}.
\end{align*}
By $F(x,\theta) = A(\theta) x +b$ and direct calculations we get
\begin{align*}
\partial_\theta F(\hat{x},\theta) = A^\prime(\theta) \hat{x}.
\end{align*}
Now rearrange and plugging in $\hat{x} = F^{-1}(x^\theta(a),\theta)$ gives:
\begin{align*}
\partial_\theta w_h^\theta(x)|_{x = x^\theta(a)} = -D_x w_h^\theta(x^\theta(a)) \cdot A^\prime(\theta) F^{-1}(x^\theta(a),\theta).
\end{align*}
%Now we only have to evaluate in $x = x^\theta$ and estimate
%\begin{align}
%\frac{\d}{\d \theta} w_h^\theta(x)|_{x=x^\theta} = &\ -\nabla_{\Ga_h^\theta} w_h^\theta(x^\theta) A^\prime(\theta) A^{-1}(\theta) (x^\theta-b).
%\end{align}
By construction we note that
\begin{equation}
\label{eq:Ascaling}
|A'(\theta)|
= \mathcal{O}\!\left(
\begin{bmatrix}
| & | & |\\
h^2 & h^2 & h\\
| & | & |
\end{bmatrix}
\right),
\quad
|A^{-1}(\theta)|
= \mathcal{O}\!\left(
\begin{bmatrix}
- & 1/h & -\\
- & 1/h & -\\
- & 1 & -
\end{bmatrix}
\right),
\quad
x^\theta-b=
\begin{bmatrix}
h\\
h\\
h\\
\end{bmatrix}.
\end{equation}
Thus bound the contributions from the mapping $F$ as 
\begin{equation}
\label{eq:bound_partial_theta_F}
\partial_\theta F(F^{-1}(x^\theta(a),\theta),\theta) = A^\prime(\theta) A^{-1}(\theta) (x^\theta(a)-b) = \mathcal{O}(h^2).
\end{equation}
The discrete gradient can be unlifted via norm equivalence arguments \eqref{eq:norm_equiv}.

We determine bounds for $\frac{\d}{\d  \theta} x^\theta(a)$ by deriving an explicit formula for $x^\theta(a)$. We reduce the analysis, assuming that we use newest-vertex bisection as the refinement-strategy. This implies that each refined element has at most one node, which has to be lifted. 
Let $A,B$ be fixed nodes, and $C^\theta$ the new node, define the vectors:
\begin{equation*}
	\begin{alignedat}{3}
		v_a \;=\; a - x^0(a),
		&\quad& \text{the lift vector, along which the points $x^\theta(a)$ are lifted onto $\Ga$,}
		\\[1.5ex]
		\begin{rcases}
			v_{AB} &= B - A,\\
			v_{AC^\theta} &= C^\theta - A,
		\end{rcases}
		&\quad& \text{the element vectors, which span the $\theta$-element $\Delta ABC^\theta$.}
	\end{alignedat}
\end{equation*}

Using these vectors, we parametrize the fixed projection line $\overline{x^0(a)\, a}$ and the moving element $\Delta ABC^\theta$. This allows us to determine the intersection $x^\theta(a)$ by solving \begin{equation*}
	x^0(a) + \lambda(\theta) v_a = A + \mu(\theta) v_{AB} + \xi(\theta) v_{AC^\theta}.
\end{equation*} 
The unique solution of this linear system is given by Cramer's rule 
\begin{equation*}
	\lambda(\theta) = \frac{\det (A-x_0(a), v_{AB}, v_{AC^\theta})}{\det (v_a, v_{AB}, v_{AC^\theta})}.
\end{equation*} 
This gives an explicit formula for the intersection point $x^\theta (a) = x^0(a) + \lambda(\theta) v_a$.

\begin{figure}[htbp]
	\begin{tikzpicture}[scale=2]
		% Draw the gentle "elliptical" curve from (2,0) to (0,1)
		\draw[thick, smooth, samples=200, domain=0:90]
		plot({2*cos(\x)}, {sin(\x)});
		\coordinate (top) at (0.5,1);
		\node[above] at (top) {$\Ga$};
		% Draw and label the horizontal baseline from A(0) to C(0)
		\draw[black] (0,0) node[left]{$C^0$} -- (2,0);
		\fill (0,0)    circle[radius=0.025];
		
		% Draw and label the new edge from A(1) to C(1)
		\draw[black] (0,1) node[left]{$C^1$} -- (2,0) node[right]{$A^0 = A^1$};
		\fill (2,0)    circle[radius=0.025];
		\fill (0,1)    circle[radius=0.025];
		
		% Draw and label the new edge from A(1) to C(1)
		\draw[black, dashed] (0,0.5) node[left]{$C^\theta$} -- (2,0);
		\fill (0,0.5)    circle[radius=0.025];
		
		% Sample several points and draw inward normals analytically
		\foreach \t in {15,30,45,67,80,90} {
			% Compute point on the curve
			\pgfmathsetmacro{\xi}{2*cos(\t)}
			\pgfmathsetmacro{\yi}{sin(\t)}
			% Compute tangent vector components: (-2*sin(t), cos(t))
			\pgfmathsetmacro{\tx}{-2*sin(\t)}
			\pgfmathsetmacro{\ty}{ cos(\t)}
			% Inward normal vector: (ty, -tx)
			\pgfmathsetmacro{\nx}{\ty}
			\pgfmathsetmacro{\ny}{-\tx}
			% Intersection with baseline y=0: s0 = -yi/ny
			\pgfmathsetmacro{\szero}{-\yi/\ny}
			\pgfmathsetmacro{\xzero}{\xi + \szero*\nx}
			% Draw the normal from curve to baseline
			\draw[gray] (\xi,\yi) -- (\xzero, 0);
			% For t=45, label a, x^0(a), and x^1(a)
			\ifnum\t=45
			% Label point on curve
			\node[above] at (\xi,\yi) {$a$};
			% Label intersection with baseline
			\node[below] at (\xzero,0) {$x^0(a)$};
			% Intersection with edge from (0,1) to (2,0): solve system
			% s1 = (2 - xi - 2*yi) / (nx + 2*ny)
			\pgfmathsetmacro{\sone}{(2 - \xi - 2*\yi)/(\nx + 2*\ny)}
			\pgfmathsetmacro{\uone}{(\xi + \sone*\nx)/2}
			\pgfmathsetmacro{\xone}{2*\uone}
			\pgfmathsetmacro{\yone}{1 - \uone}
			\node at (\xone+0.3,\yone+0.075) {$x^1(a)$};
			\fill (\xi,\yi)    circle[radius=0.025];
			\fill (\xzero,0)   circle[radius=0.025];
			\fill (\xone,\yone) circle[radius=0.025];

			\node[gray] at (0.45*\xzero+0.55*\xone-0.3,0.55*\yone-0.075) {$x^\theta(a)$};
			\fill[gray] (0.45*\xzero+0.55*\xone,0.55*\yone)   circle[radius=0.025];
			\fi
		}
	\end{tikzpicture}
	~
	\begin{tikzpicture}[scale=2]
		% Draw the gentle "elliptical" curve from (2,0) to (0,1)
		\draw[thick, smooth, samples=200, domain=0:90]
		plot({2*cos(\x)}, {sin(\x)});
		\coordinate (top) at (0.5,1);
		\node[above] at (top) {$\Ga$};
		% Draw and label the horizontal baseline from A(0) to C(0)
		\draw[black] (0,0) node[left]{$C^0$} -- (2,0);
		\fill (0,0)    circle[radius=0.025];
		
		% Draw and label the new edge from A(1) to C(1)
		\draw[black] (0,1) node[left]{$C^1$} -- (2,0) node[right]{$A^0 = A^1$};
		\fill (2,0)    circle[radius=0.025];
		\fill (0,1)    circle[radius=0.025];
		
		% Draw and label the new edge from A(1) to C(1)
		\draw[black, dashed] (0,0.5) node[left]{$C^\theta$} -- (2,0);
		\fill (0,0.5)    circle[radius=0.025];
		
		% vector v_{AC}
		\draw [thick, <-] (0,0.5) -- (2, 0) node[pos=0.45, above]{$v_{AC^\theta}$};
		
		% Sample several points and draw inward normals analytically
		\foreach \t in {15,30,45,67,80,90} {
			% Compute point on the curve
			\pgfmathsetmacro{\xi}{2*cos(\t)}
			\pgfmathsetmacro{\yi}{sin(\t)}
			% Compute tangent vector components: (-2*sin(t), cos(t))
			\pgfmathsetmacro{\tx}{-2*sin(\t)}
			\pgfmathsetmacro{\ty}{ cos(\t)}
			% Inward normal vector: (ty, -tx)
			\pgfmathsetmacro{\nx}{\ty}
			\pgfmathsetmacro{\ny}{-\tx}
			% Intersection with baseline y=0: s0 = -yi/ny
			\pgfmathsetmacro{\szero}{-\yi/\ny}
			\pgfmathsetmacro{\xzero}{\xi + \szero*\nx}
			% Draw the normal from curve to baseline
			\draw[gray] (\xi,\yi) -- (\xzero, 0);
			% For t=45, label a, x^0(a), and x^1(a)
			\ifnum\t=45
			% Label point on curve
			\node[above] at (\xi,\yi) {$a$};
			% Label intersection with baseline
			\node[below] at (\xzero,0) {$x^0(a)$};
			% Intersection with edge from (0,1) to (2,0): solve system
			% s1 = (2 - xi - 2*yi) / (nx + 2*ny)
			\pgfmathsetmacro{\sone}{(2 - \xi - 2*\yi)/(\nx + 2*\ny)}
			\pgfmathsetmacro{\uone}{(\xi + \sone*\nx)/2}
			\pgfmathsetmacro{\xone}{2*\uone}
			\pgfmathsetmacro{\yone}{1 - \uone}
			%				\node at (\xone+0.3,\yone+0.075) {$x^1(a)$};
			\fill (\xi,\yi)    circle[radius=0.025];
			\fill (\xzero,0)   circle[radius=0.025];

			% vector v_a
			\draw [thick, <-] (\xi,\yi) -- (\xzero, 0) node[pos=0.65, right]{$v_a$};
			\fi			
		}
	\end{tikzpicture}
	\caption{Sketch of the two-dimensional situation. With the element $\overline{AC^\theta}$ of $\Ga_h^\theta$ for $\theta \in [0,1]$, the point $x^\theta(a)$, and the vectors $v_a$ and $v_{AC^\theta}$.}
	\label{fig:theta construction}
\end{figure}

In order to compute the derivative $\frac{\d}{\d \theta} x^\theta(a) = \frac{\d}{\d \theta} \lambda(\theta) v_a$, we need to determine the derivative of $\lambda(\theta)$. %to quantify the rate of change of $x^\theta$ with respect to $\theta$.  
Let $N(\theta) := \det\bigl(A-x_0(a),v_{AB},v_{AC^\theta}\bigr)$ and $D(\theta) := \det\bigl(v_a,v_{AB},v_{AC^\theta}\bigr)$.
Then, since only one of the vectors is $\theta$-dependent, we obtain
\begin{align*}
	\frac{\d}{\d \theta} N(\theta)
	= \det \bigg( A - x_0(a), v_{AB}, \frac{\d}{\d \theta} v_{AC^\theta} \bigg),
	\qquad
	\frac{\d}{\d \theta} D(\theta)
	= \det \bigg( v_a, v_{AB}, \frac{\d}{\d \theta} v_{AC^\theta} \bigg).
\end{align*}
By the element size and the bound $\|d\|_{L^\infty} \leq ch^2$, we know that $A-x^0(a) = \mathcal{O}(h)$,  $v_a = \mathcal{O}(h^2)$ and $v_{AB} = \mathcal{O}(h) = v_{AC^\theta}$, and $\frac{\d}{\d \theta} v_{AC^\theta} = \frac{\d}{\d \theta} C^\theta = C^1 -C^0 =\mathcal{O}(h^2)$.
Therefore we have
\begin{align*}
	N(\theta) =  \mathcal{O}(h^3), \quad \frac{\d}{\d \theta} N(\theta) = \mathcal{O}(h^4) , \quad \text{and} \quad
	D(\theta) = \mathcal{O}(h^4), \quad \frac{\d}{\d \theta} D(\theta) = \mathcal{O}(h^5) , 
\end{align*}
%, $N''(\theta) = \mathcal{O}(h^4)$, $D''(\theta) = \mathcal{O}(h^5)$, 
In particular, note that $\nu(a)$ is collinear with $v_a$, and $v_{AB} \times v_{AC^\theta}$ is collinear to $\nu_h(x(a))$, hence, using that $\|\nu-\nu_h\|_{L^\infty} = \mathcal{O}(h)$, the angle between $v_a$ and $v_{AB} \times v_{AC^\theta}$ is small. Therefore, the denominator $D(\theta)$ is in fact of forth order $D(\theta) \sim \mathcal{O}(h^4)$. 

We then have
\begin{align*}
\frac{\d}{\d \theta} \lambda(\theta)  = \frac{\frac{\d}{\d \theta} N(\theta)D(\theta)-N(\theta)\frac{\d}{\d \theta} D(\theta)}{D(\theta)^2} = \mathcal{O}(1) .
\end{align*}

We therefore have shown the bound
\begin{equation*}
	\frac{\d}{\d \theta} x^\theta = \mathcal{O}(h^2) .
\end{equation*}

Altogether, we obtain the estimate
\begin{align*}
	&\ \| (\mathcal{I}_{n-1}^{\comesh} u^{n-1}_h)^{\ell^{\comesh}} - (u^{n-1}_h)^{\ell^{n-1}} \| = \|(w_h^1)^{\ell^1}(a)-(w_h^0)^{\ell^0}(a)  \|  \\
	&\ \leq \int_0^1 \Big\| \big(\nabla_{\Ga_h^\theta} w_h^\theta \big)^{\ell^\theta} \Big\| \Big\|\partial_\theta x^\theta(a)-A^\prime(\theta) F^{-1}(x^\theta(a),\theta) \Big\| \d \theta \\
	&\ \leq  C h^2 \int_0^1 \big\| \nabla_{\Ga_h^\theta} w_h^\theta \big\| \d \theta 
	\leq C h^2 \|\nabla_{\Ga_h^0} w_h^0 \|_{L^2(\Ga_h^0)} =  C h^2 \|\nabla_{\Ga_h^{n-1}} u_h^{n-1} \|_{L^2(\Ga_h^{n-1})} ,
\end{align*}
where the final inequality follows by the $\theta$-uniform norm equivalence \cite[Lemma~4.3]{KLLP_2017}.

The second term of \eqref{eq:ind_coarse_not_simple} is handled analogously. First note that the time dependent $\frac{t^n-t}{\tau^n} \leq 1$, so we ignore this factor. As before we insert zeros of the form $\pm \nabla_\Ga [(\mathcal{I}_n^{\comesh} \IntRef u^{n-1}_h)^{\ell^{\comesh}}] \pm \nabla_\Ga[(\mathcal{I}_n^{\comesh} u^{n-1}_h)^{\ell^{\comesh}}]$, upon the triangle inequality we obtain the term from the coarsening indicator  \eqref{eq:indicator - coarse}.

We again apply the fundamental theorem of calculus and take the gradient afterwards. We simplify directly by applying the product rule and using that fact that $\nabla_{\Ga_h^\theta} w_h^\theta$ is elementwise constant, thus the tangential gradient of its lift vanishes. We remark that the surface gradient $\nabla_\Ga$ always differentiates with respect to $a$. This gives
\begin{align*}
\nabla_\Ga (w_h^1)^{\ell^1}\!(a)\!-\nabla_\Ga (w_h^0)^{\ell^0}\!(a)\! &= \int_0^1 \nabla_\Ga \Big( \frac{\d}{\d  \theta} (w_h^\theta)^{\ell^\theta}\!(a)\! \Big) \d \theta \\
&= \int_0^1 \big(\nabla_{\Ga_h^\theta} w_h^\theta\big)^{\ell^\theta} \nabla_\Ga \bigg(\frac{\d}{\d  \theta} x^\theta\!(a)\! - A^\prime(\theta) A^{-1}(\theta)(x^\theta\!(a)\!-b) \bigg)
\end{align*}
We employ previous results to simplify the tangential gradients. The contributions from the standard finite element map are bound by 
\begin{equation}
\nabla_\Ga \big(A^\prime(\theta) A^{-1}(\theta)(x^\theta(a)-b)\big) = A^\prime(\theta) A^{-1}(\theta) \nabla_\Ga x^\theta(a) = \mathcal{O}(h).
\end{equation}
where we used previous results \eqref{eq:Ascaling} on $A(\theta)$. Further from the elementwise $\|D_a (a-x(a))\|_{L^\infty} \leq ch$ \cite{elliot_ranner_13} on the corresponding curved element, %page 8 last line
the explicit description of $x^\theta(a) = a-\lambda(\theta) v_a$, and the trivial bounds $\lambda(\theta) \leq 1$ and $\|P\| = 1$ we are able to bound the tangential jacobian of $x^\theta(a)$
\begin{align*}
\|\nabla_\Ga x^\theta\!(a)\!\|_{L^\infty}\! = \|P-\!\lambda(\theta)\! \nabla_\Ga (a-x^0\!(a)\!) \|_{L^\infty}\! \leq \|P\|_{L^\infty}+\|\nabla_\Ga (a-x^0\!(a)\!) \|_{L^\infty}\!  \leq 1+ch,
\end{align*}
which concludes the estimate.
Finally we analyse the contributions of the lift map directly, again using its explicit description
\begin{equation}
\bigg\|\nabla_\Ga \Big(\frac{\d}{\d  \theta} x^\theta(a)\Big)\bigg\|_{L^\infty} = \bigg\|\frac{d}{d \theta}\lambda(\theta) \nabla_\Ga (a-x^0(a))\bigg\|_{L^\infty} \leq ch.  
\end{equation}
In total this shows \eqref{eq:indicator - coarse}. 

\subparagraph{Multiple refinements}
By restricting the analysis to one-level refinements only, we have the convenient property that new nodes are unlifted to the midpoint of discrete elements, cf.~Figure~\ref{fig:nodal_values_of_common_triangulation}. However, in the general setting, multiple consecutive refinements of an element may be required.

We resolve such multi-level refinements as follows:
Introduce intermediate meshes, where each intermediate mesh is constructed by taking the nodes of the common triangulation $\widehat{\Ga}_h^{n-1}$ and following the refinement hierarchy:
\begin{itemize}
	\item[--] lift the first level of refinements to obtain the first intermediate mesh,
	\item[--] lift the second level of refinements to obtain the second intermediate mesh,
	\item[--] etc. 
\end{itemize}
This construction allows us to use precisely the arguments from above between two consecutive intermediate meshes. 
Namely, the interpolation errors can be bound via multiple triangle inequalities by inserting the representation of $u_h^{n-1}$ and $\IntRef u_h^{n-1}$ on the intermediate meshes, combined with the $\theta$-uniformity of the resulting discrete gradients which shows the full bound.

The step-by-step refinement/coarsening and lifting, further allows for a unique common mesh construction, in particular the parent mesh cannot deteriorate by construction. This allows us to directly write the main contribution to the coarsening indicator.

In total we have the upper bound to the coarsening residual \eqref{eq:res_coarse} and thus shown that \eqref{eq:coarse_bound_new} indeed holds. 
\phantom{$\square$}
\end{proof}

\subsection{Bounding the spatial residual}
\label{section:spatial residual bound}

The decomposition \eqref{eq:res_eq} was structured such that the spatial and temporal indicators are solely defined on the discrete surface which is elementwise flat. This is an essential feature as it allows to modify results from the case of parabolic problems in flat domains by Verführt \cite{Verfuehrt1996,Verfuerth2003}. All modifications result from correctly resolving the underlying geometry of the surface PDE \eqref{eq:heat_strong}.
\begin{proposition}
\label{prop:spatial_res}
	For $0 < t^n \leq T_{\max}$ the spatial indicator $\eta_h^n$ \eqref{eq:indicator - spatial} is uniformly equivalent to the dual norm of the spatial residual $\mathcal{R}_h$ \eqref{eq:res_space}, i.e., for $t \in (t^{n-1},t^n]$,
	\begin{subequations}
	\begin{align}
		\label{eq:indicator upper bound for R_h}
		\| \mathcal{R}_h(t) \|_{{H^{-1} (\Ga)}} \leq &\ c \, \eta_h^n , \\
		\label{eq:indicator lower bound for R_h}
		\frac{1}{c} \eta_h^n \leq &\ \| \mathcal{R}_h(t) \|_{{H^{-1} (\Ga)}} .
	\end{align}
	\end{subequations}
	The constant $c > 0$ is independent of $h$ and $\tau^n$, but depends on the shape-regularity constant $\varrho^n$ of $\Ga_h^n$.
\end{proposition}

\begin{proof}
We extend the proof of \cite[Section~5]{Verfuerth2003} to surfaces. Notable modifications are required due to the geometric approximation, the use of the full $H^1$-norm and the adjusted residual.

(a) We start by showing the upper bound \eqref{eq:indicator upper bound for R_h}. 
By the definition of the fully discrete problem \eqref{eq:discreteheat} we know $\langle  \mathcal{R}_h(t), v_h \rangle = 0$ for any $v_h \in S_h$.
Thus we have $\langle  \mathcal{R}_h(t), v  \rangle = \langle  \mathcal{R}_h(t), v - I_h^{\textnormal{SZ}} v \rangle$, where $I_h^{\textnormal{SZ}} v = (\widetilde{I}_h^{\textnormal{SZ}} v)^\ell$ is the surface Scott--Zhang interpolation operator from \cite[Section~3]{APostElliptic}. The term $(\nabla_{\Ga_h^n} u_h^n, \nabla_{\Ga_h^n} ( v^{-\ell} - \widetilde{I}_h^{\textnormal{SZ}} v ))_{\Ga_h^n}$ is partially integrated on each element and simplified using the jumps and the fact that $u_h^n$ is linear. 
Finally, Hölder and Cauchy--Schwarz inequalities together with interpolation error estimates \cite[Theorem~3.2]{APostElliptic} for the triangles and \cite[Corollary~3.4]{APostElliptic} for the edges of $\Ga_h^n$  give the upper bound.

The lower bound \eqref{eq:indicator lower bound for R_h} is shown using bubble functions $\psi$ (see, e.g.,~\cite[Section~3]{Verfuehrt1996} for a concise and detailed overview) and is based on the two key relations (cf.~\cite[equation~(5.3)]{Verfuerth2003}):
\addtocounter{equation}{+1}
\begin{align} \tag{\theequation a,b} 
	\|(w^n)^{ \ell } \|_{H^1(\Ga)}  \leq c \, \eta_h^n \qquad \text{and} \qquad \langle \mathcal{R}_h(t), (w^n)^\ell \rangle \geq (\eta_h^n)^2, \label{eq:lower_bounds_proof_space}
\end{align}
with $w^n = \alpha \sum_{T \in \tilde{T}} h_T^2 \psi_T (f_h - \partial_t  \baruhtau) - \beta \sum_{S \in \tilde{S}}  h_S \psi_S  \llbracket \nabla_{T} u_h^n \cdot \mathrm{n}_S \rrbracket $.

The relation can be directly obtained from expanding the expressions, using estimates on bubble functions \cite[Section~3]{Verfuehrt1996}, and appropriate choices of the parameters $\alpha$ and $\beta$. As noted above Verfürth's proof works in the $H^1$-semi-norm. A direct calculation of the $L^2$-norm, combined with bounds in \cite[Lemma~1.3]{Verfuehrt1996}, show that we can absorb the resulting terms into the semi-norm.

We now transform all quantities to $\Ga$, then the two inequalities of \eqref{eq:lower_bounds_proof_space} yield, while applying the equivalence of norms \eqref{eq:norm_equiv} for lifted functions,
\begin{equation*}
	\|\mathcal{R}_h(t)\|_{{H^{-1} (\Ga)}} = \sup_{\substack{v \in H^1(\Ga) \\ v \neq 0}} \frac{\langle \mathcal{R}_h(t), v \rangle}{\|v \|_{H^1(\Ga)}}  \geq \frac{\langle \mathcal{R}_h(t), (w^n)^\ell \rangle}{\| (w^n)^{\ell} \|_{H^1(\Ga)} } \geq \frac{(\eta_h^n)^2}{c \, \eta_h^n} = c \, \eta_h^n .
\end{equation*}

Altogether, this proves that the spatial indicator is equivalent to the dual norm of the spatial residual.
\end{proof} 

\subsection{Bounding the temporal residual}
\label{section:temporal residual bound}

To bound the temporal residual, we again follow \cite[Section~6 and 7]{Verfuerth2003}, always carefully extending the Euclidean results to surfaces, and highlighting the insights and modifications required by geometric approximations, and the coarsening term.
\begin{proposition}
\label{prop:temp}
	(a) The temporal indicator $\eta_\tau^n$ \eqref{eq:indicator - temporal} and the temporal residual $\mathcal{R}_\tau$ \eqref{eq:res_temp} satisfy, for any $0 < t^n \leq T_{\max}$,
	\begin{subequations}
	\begin{equation}
		\label{eq:indicator upper bound for R_tau}
		\|R_\tau\|_{L^2(t^{n-1},t^n;H^{-1}(\Ga))} \leq c \, (\tau^n)^{\frac{1}{2}}  \, \eta_{\tau}^n .
	\end{equation}

	(b) The temporal indicator $\eta_\tau^n$ \eqref{eq:indicator - temporal}  and the error $u-u_{h,\tau}$ satisfy, for any $0 < t^n \leq T_{\max}$,
	\begin{align}
	\label{eq:final_temporal_bound_2}
%	\addtocounter{equation}{+1}
		(1+h) (\tau^n)^{\frac{1}{2}} \eta_\tau^n \leq &\ c \Big(\|u- u_{h,\tau}\|_{X(t^{n-1},t^n;\Ga)} + \|f-f_h^\ell \|_{L^2(t^{n-1},t^n;H^{-1}(\Ga))} \\
		&\ \hphantom{c \Big( }+ (\tau^n)^{\frac{1}{2}} (\mathcal{G}^n + \eta_c^n)\Big) .
  	\end{align}
	\end{subequations}
	The constant $c > 0$ is independent of $h$ and $\tau^n$, but depends on the shape-regularity constant $\varrho^n$ of $\Ga_h^n$, and on $\Ga$. 
\end{proposition}

\begin{proof}
(a) Starting with the upper bound \eqref{eq:indicator upper bound for R_tau}, we remark that $\baruhtau(t)$ is piecewise linear in time, thus we rewrite the error in the temporal residual \eqref{eq:res_temp}, for any $t \in (t^{n-1},t^n]$, as
\begin{align*}
	u_h^n - \baruhtau = \left(1-\frac{t-t^{n-1}}{\tau^n}\right)(u_h^n- \IntRef u_h^{n-1}) .
\end{align*}
Then, via norm equivalence \eqref{eq:norm_equiv}, we have
\begin{align*}
	\|\mathcal{R}_\tau (t)\|_{H^{-1} (\Ga)} 
%	= \sup_{\substack{v \in H^1(\Ga) \\ v \neq 0}} \frac{\langle \mathcal{R}_\tau (t), v \rangle}{\|v \|_{H^1(\Ga)}} 
	&\leq \sup_{\substack{v \in H^1(\Ga) \\ v \neq 0}} \frac{\|\nabla_{\Ga_h^n} (u_h^n-  \baruhtau)\|_{L^2(\Ga_h^n)} \|\nabla_{\Ga_h^n} v^{-\ell}\|_{L^2(\Ga_h^n)}}{c \|v^{-\ell} \|_{H^1(\Ga_h^n)}} \\
	&\leq c \|\nabla_{\Ga_h^n} (u_h^n-  \baruhtau)\|_{L^2(\Ga_h^n)}.
\end{align*}
Integrating this inequality over $(t^{n-1},t^n]$,and recalling \eqref{eq:indicator - temporal} yields the upper bound.

(b) For the bound between the temporal indicator and the error \eqref{eq:indicator lower bound to error + R_h}, we follow the proof in \cite[Section~7]{Verfuerth2003}, the computation is straightforward, thus we focus on the differences which arise mostly from the additional geometric and coarsening residual and due to the natural $H^1$-norm.

Similar to \cite{Verfuerth2003} we seek a test function $w$ which gives a direct relation between temporal indicator and temporal residual, due to the full $H^1$-norm it will be beneficial to take a similar test function as Verfürth but subtract its mean. Let $\overline{c} := \frac{1}{|\Ga|} \int_\Ga (u_h^n - \baruhtau)^\ell$ then
\begin{equation}
\label{eq:two expressions for res}
	\begin{aligned}
		\frac{\tau^n}{3} (\eta_\tau^n)^2 = &\ \int_{t^{n-1}}^{t^n} \langle \mathcal{R}_\tau (t), (u_h^n - \baruhtau)^\ell - \overline{c} \rangle \d t, \\
	\end{aligned}
\end{equation}
Note that the temporal residual \eqref{eq:res_temp} is invariant under constant shifts.

Now using the residual decomposition \eqref{eq:res_eq} and \eqref{eq:two expressions for res}, yields
\begin{align*}
	\frac{\tau^n}{3} (\eta_\tau^n)^2 
	\leq &\ \int_{t^{n-1}}^{t^n} \langle \mathcal{R}_\tau (t), (u_h^n - \baruhtau)^\ell - \overline{c} \rangle \d t \\
%	\leq &\ c_2 \Big( \frac{\tau^n}{3} \Big)^{1/2} \eta_\tau^n \, \Big( \|u- u_{h,\tau}\|_{X(t^{n-1},t^n;\Ga)} + \|f-f_h^\ell \|_{L^2(t^{n-1},t^n;H^{-1}(\Ga))} \\
%	&\ \hphantom{c_2 \Big( \frac{\tau^n}{3} \Big)^{1/2} \eta_\tau} + c_3 (\tau^n)^{1/2} \eta_h^n + \|\mathcal{R}_{\textnormal{g}}\|_{L^2(t^{n-1},t^n;H^{-1}(\Ga))} \Big),
	\leq &\ c (1+C_{\textnormal{P}}^2)^{\frac{1}{2}} \, \Big( \frac{\tau^n}{3} \Big)^{\frac{1}{2}} \eta_\tau^n \, \Big(\|u- u_{h,\tau}\|_{X(t^{n-1},t^n;\Ga)} + \|f-f_h^\ell \|_{L^2(t^{n-1},t^n;H^{-1}(\Ga))} \\
	&\ \qquad + (\tau^n)^{\frac{1}{2}} \, \eta_h^n + \|\mathcal{R}_{\textnormal{g}}\|_{L^2(t^{n-1},t^n;H^{-1}(\Ga))} +\|\Rcoarse\|_{L^2(t^{n-1},t^n;H^{-1}(\Ga))} \Big),
\end{align*}
We immediately used Proposition~\ref{prop:spatial_res}, and the bound 
\begin{align*}
\| (u_h^n - \baruhtau)^\ell\!-\!\overline{c}\|_{L^2(t^{n-1},t^n;H^{1}(\Ga))}  &\leq (1\!+\!C_{\textnormal{P}}^2)^{\frac{1}{2}} \Big(\int_{t^{n-1}}^{t^n} \|\nabla_\Ga (u_h^n - \baruhtau)^\ell\|^2_{L^2(\Ga)}\Big)^{\frac{1}{2}} \\
&\leq (1\!+\!C_{\textnormal{P}}^2)^{\frac{1}{2}} \Big( \frac{\tau^n}{3} \Big)^{\frac{1}{2}} \Big\|\nabla_\Ga \big((u_h^n - \IntRef u_h^{n-1})^\ell \big)\Big\|_{L^2(\Ga)} \\
&\leq c(1\!+\!C_{\textnormal{P}}^2)^{\frac{1}{2}} \Big( \frac{\tau^n}{3} \Big)^{\frac{1}{2}} \eta_\tau^n
\end{align*}
obtained by a Poincaré inequality (with the constant $C_{\textnormal{P}}>0$) for the mean-free test function $(u_h^n - \baruhtau)^\ell - \overline{c}$, direct calculation of \eqref{eq:fully discrete solution definition}, and norm equivalence \eqref{eq:norm_equiv}.

The geometric residual term is bounded in terms of the spatial and temporal residual via the time-integrated estimate \eqref{eq:geo_res_bound}.
Rearranging the terms,  and using $h \leq h_0$  yields the intermediate result
\begin{equation}
\label{eq:indicator lower bound to error + R_h}
	\begin{aligned}
		(1+ h ) (\tau^n)^{\frac{1}{2}} \, \eta_\tau^n \leq &\ c \Big(\|u- u_{h,\tau}\|_{X(t^{n-1},t^n;\Ga)} + \|f-f_h^\ell \|_{L^2(t^{n-1},t^n;H^{-1}(\Ga))} \\
		&\ \phantom{c \Big( } +  (\tau^n)^{\frac{1}{2}} \mathcal{G}^n + (\tau^n)^{\frac{1}{2}}\eta_c^n  + (\tau^n)^{\frac{1}{2}} \eta_h^n \Big) .
	\end{aligned}
\end{equation}

\paragraph{Bounding the spatial indicator}
\label{lower_bound_spatial}
We now bound the spatial indicator $(\tau^n)^{\frac{1}{2}} \eta_h^n$ by the error,  oscillation term,  $\mathcal{G}^n$, and $\eta_c$ , which shows \eqref{eq:final_temporal_bound_2}.

To bound the spatial term we follow \cite[equation~(7.1)--(7.4)]{Verfuerth2003}, but the new geometric and coarsening residuals need to be carefully estimated during the decomposition. 

We take the second inequality from \eqref{eq:lower_bounds_proof_space}, multiply the expression by the factor $(\alpha +1)\left(\frac{t-t^{n-1}}{\tau^n}\right)^\alpha$ and integrate over $(t^{n-1},t^n]$ where $\alpha \geq 0$. The idea in \cite{Verfuerth2003} is to exchange the spatial residual by the decomposition \eqref{eq:res_eq}, and then bounding each term separately, for $h \leq h_0$ sufficiently small. For the present case these estimates are nearly identical, and thus omitted here.

The geometric residual is bound by the temporal and spatial residuals, and the geometric indicator \eqref{eq:indicator - geo}, using Lemma~\ref{lemma:geo}. 

We obtain an estimate involving the spatial and temporal indicators, the graph norm of the error, and the oscillation:
\begin{align*}
	\tau^n (\eta_h^n)^2 \leq &\ c_\alpha \, (\tau^n)^{\frac{1}{2}} \eta_h^n \Big(\|u-u_{h,\tau} \|_{X(t^{n-1},t^n; \Ga)} + \|f-f_h^\ell\|_{L^2(t^{n-1},t^n;H^{-1} (\Ga))} \\
	&\  \qquad  + c \, (\tau^n)^{\frac{1}{2}} \mathcal{G}^n + c \, (\tau^n)^{\frac{1}{2}} \eta_c  + c  h  (\tau^n)^{\frac{1}{2}} \eta_h^n \Big)  + c_\alpha (1+ h ) (\tau^n)^{\frac{1}{2}} \eta_\tau^n .
\end{align*}
Here the $\alpha$ subscript denotes an $\alpha$ dependency of the constant.
 
Using \eqref{eq:indicator lower bound to error + R_h} to bound the temporal indicator, simplifications and including the factor $(1+h)$ results in the inequality,
\begin{align*}
	(1+ h) (\tau^n)^{\frac{1}{2}} \eta_h^n \leq &\ c_\alpha \big(\|u-u_{h,\tau} \|_{X(t^{n-1},t^n; \Ga)} + \|f-f_h^\ell\|_{L^2(t^{n-1},t^n;H^{-1} (\Ga))} \\
	&\ \phantom{c_\alpha \big( }  + c \, (\tau^n)^{\frac{1}{2}} \mathcal{G}^n + c \, (\tau^n)^{\frac{1}{2}} \eta_c  \big) +  c_\alpha (1+ h ) (\tau^n)^{\frac{1}{2}} \eta_h^n ,
\end{align*}
which only contains the spatial indicator, the error,  the geometric $\mathcal{G}$, the coarsening,  and the oscillation term.
 
The spatial indicator appears on both sides of the inequality, however an appropriate choice of $\alpha \geq 0$, see \cite{Verfuerth2003}, (which, compared to \cite{Verfuerth2003}, additionally depends on the constant from the norm equivalence, but does not depend on $h$), allows us to absorb the spatial indicator from the left-hand side, while retaining a positive constant on the right-hand side. Thus the spatial residual is bounded by the oscillation, the geometric $\mathcal{G}$, the coarsening, and the error term. 

This directly infers the same for the temporal indicator by \eqref{eq:indicator lower bound to error + R_h}. 
\end{proof}

\subsection{Proof of Theorem~\ref{thm:upper_lower}}
\label{section:thm proof}
~ \newline
(a) We will now prove the global upper bound \eqref{eq:reliability estimate}. 
Altogether, the upper bound follows from the residual decomposition \eqref{eq:res_eq} combined with the previously shown upper bounds (marked by (a)). 
We have, for any $0 \leq t^n \leq T_{\max}$, (again dropping the omnipresent $\Ga$ within norms)
\begin{align*}
	\|\mathcal{R}(u_{h,\tau}) \|_{L^2(0,t^n; H^{-1} )}
	\leq &\ \|\mathcal{R}_h \|_{L^2(0,t^n; H^{-1} )} + \|\mathcal{R}_\tau \|_{L^2(0,t^n; H^{-1} )} + \|\!\osc(f)\|_{L^2(0,t^n; H^{-1})} \\
	&\ +\|\Rcoarse \|_{L^2(0,t^n; H^{-1} )}  + \|\mathcal{R}_{\textnormal{g}} \|_{L^2(0,t^n; H^{-1} )}, \\
\end{align*} 
We will now combine the upper bounds to the above terms (breaking $(0,t^n]$ into the subintervals $(t^{n-1},t^n]$). That is, respectively, 
the spatial bound \eqref{eq:indicator upper bound for R_h} integrated over time, 
the temporal bound \eqref{eq:indicator upper bound for R_tau}, 
the coarsening bound \eqref{eq:coarse_bound_new}, 
and the geometric bound \eqref{eq:geo_res_bound}.
Finally, plugging into \eqref{eq:res_equiv_2} shows the upper bound.

(b) The local lower bound \eqref{eq:efficiency estimate} follows by recalling \eqref{eq:indicator}, and using the lower bounds \eqref{eq:indicator lower bound for R_h} and \eqref{eq:final_temporal_bound_2}, combined with the absorption of the geometric residual by \eqref{eq:geo_res_bound} (excluding $\mathcal{G}^n$).

This completes the proof of Theorem~\ref{thm:upper_lower}. \qed

\section{Adaptive algorithm for parabolic surface PDEs}
\label{adaptivity}

In this chapter we provide a simple adaptive routine (Algorithm~\ref{alg:base_alg}) which guarantees a global error bound by a given tolerance, assuming %convergence with respect to some given tolerances given that
the algorithm terminates. We highlight the main differences which arise when constructing an adaptive algorithm for parabolic surface PDEs. Adaptive algorithms for parabolic PDEs on euclidean domains are well studied. Early contributions \cite{Dupont1982MeshMF} focussed on necessary properties for adaptation in space and time and in \cite{Bieterman1982TheFE} they proposed an adaptive algorithm based on a-posteriori error estimation.

We introduce a similar algorithm to \cite[Algorithm~3.2]{chen2004feng_adap_flat}, since their indicators where also based on the a posteriori error estimates derived by Verfürth \cite{Verfuerth2003}, i.e.\ there is the split in temporal and spatial errors. We want to highlight the effects of the curved domain, in particular related to the lift procedure, for a rather simple adaptive approach, as we are not aware of any parabolic adaptive algorithm for surface PDEs. 

We note that there are more modern adaptive algorithms showing convergence and optimality for parabolic PDEs in flat domains like \cite{AdaptiveAlgHeatEq}. We will not focus on proving convergence or optimality but if we can ensure the shape-regularity property throughout time, using for example the procedure described in \cite{Bonito2013AFEMFG}, and handle the slightly changing domain the results from the flat case should be transferable to surfaces.

In the same manner as in \cite{chen2004feng_adap_flat} or \cite{AdaptiveAlgHeatEq} we control the spatial \eqref{eq:indicator - spatial} and temporal indicators \eqref{eq:indicator - temporal} with two different refinements. Temporal refinements reduce the time step-size by some factor in $(0,1)$, this is indifferent to the flat case. The spatial refinement controls the error of the elliptic subproblem \eqref{eq:discreteheat} in each timestep plus the contribution by the geometric indicator \eqref{eq:indicator - geo}. This is done with the common solve--estimate--mark--refine routine (see \cite{Doerfler1996}). But we have to be careful with new nodes created during refinement. First the prior solution $u_h^{n-1}$ has to be evaluated at the refined nodes to obtain $\IntRef u_h^{n-1}$, as described in Section~\ref{sec_sub:ref_int}. Afterwards the nodes have to be lifted before solving again to keep the property that our triangulation is an interpolation of $\Ga$ and to fit in the framework introduced in Section~\ref{sec:res_coarse} to show the bounds for the coarsening indicator \eqref{eq:indicator - coarse}.

If all error indicators are smaller than the tolerance we store the solutions. We reinitialize for the next step, coarsen in time by initializing the new time-step size as a multiple of the last time-step size, and coarsen spatially via the coarsening indicator. Notice that temporal refinements increase (by decreasing the time step size) and spatial refinements decrease the coarsening indicator (by refinements of coarse elements), which is why we recheck the size of the coarsening indicator after temporal refinements.

%Compact version:
\begin{algorithm}[H]
\caption{Adaptive Algorithm for Parabolic Surface PDEs}
\label{alg:base_alg}
\begin{algorithmic} 
\STATE Given $\Ga_h^0$, $\tau^0$, $u_h^0$, $\TOL$, and $\theta_{\textnormal{c}} = 1$.
\WHILE{$t < T_{\textnormal{end}}$}
\STATE Coarsen space as long as $(\eta_c^n)^2 \leq \theta_{\textnormal{c}} \TOL$; \, $\Ga_h^n \leftarrow \textnormal{coarse}(\Ga_h^{n-1})$
\WHILE{$(\eta^n _{h})^2 + (\mathcal{G}^n)^2 \geq \TOL$} %(2)
\STATE Solve \eqref{eq:discreteheat}; Estimate \eqref{eq:indicator}; Mark, Refine \& Lift (Section~\ref{ch:impl})
\ENDWHILE
\vspace*{0.15cm}
\IF{$(\eta_{\tau}^n)^2 \geq \TOL$}
\STATE $t \leftarrow t-\tau$; $\tau \leftarrow \frac{\tau}{2}$; $\eta_{h}^n \leftarrow \infty$; Return to spatial loop.
\IF{$(\eta_{\textnormal{c}}^n)^2 \geq \TOL$}
\STATE $\theta_{\textnormal{c}} \leftarrow \frac{\theta_{\textnormal{c}}}{2}$; Return to coarsening.
\ELSE
\STATE Store solution.
\STATE Coarsen time $\tau = 2 \tau$; Reset $\theta_{\textnormal{c}} = 1$.
\ENDIF
\ENDIF
\ENDWHILE
\end{algorithmic}
\end{algorithm}

\section{Implementation}
\label{ch:impl}
We give a brief description of marking, refining and coarsening strategies,  and on the implementation of the refinement interpolation.  

\subsection{Marking, refining and coarsening}
\label{section:mark refine coarsen}
There are a variety of marking strategies, we will briefly recall two from Dörfler \cite{Doerfler1996}. We assume that the local error indicators are determined elementwise for all $T$ of $\Ga_h^n$, following the formulas in \eqref{eq:indicator}.

Now a subset of triangles is chosen fulfilling a marking strategy. Collect the set of marked elements in the set $\mathcal{M} \subset \Ga_h^n$.  Given a parameter $\theta \in (0,1)$ and the maximal indicator $\eta_{\max} := \max_{ T \in \Ga_h^n }  \eta(T)$ and global indicator $\eta_{\Ga_h^n}^2 = \sum_{ T \in \Ga_h^n } \eta(T)^2$, consider two \textit{refinement} criteria:
\begin{alignat*}{3}
	&\ \textnormal{- bulk criterion:} & \quad & \text{mark a subset } \mathcal{M} \subset \Ga_h^n \text{ such that } \eta (T) \geq \theta \, \eta_{\max} \quad \forall T \in \mathcal{M}, \\
	&\ \textnormal{- D\"orfler criterion:} & \quad & \text{mark a subset } \mathcal{M} \subset \Ga_h^n \text{ such that } 
			\sum_{T \in \mathcal{M}} \eta (T)^2 \geq (1- \theta) \, \eta_{\Ga_h^n}^2 .
\end{alignat*}

Both criteria mark nearly all elements if $\theta \approx 0$ and refine very few elements if $\theta \approx 1$. The D\"orfler criterion \cite[Section 4.2]{Doerfler1996} for flat elliptic problems was proven to be uniformly convergent for an elliptic model problem in \cite{Doerfler1996}. Furthermore, the D\"orfler criterion is implemented such that elements with the largest local error are added to $\mathcal{M}$ until the inequality holds. The bulk criterion is also known as the maximum strategy.

Analogously, for \textit{coarsening} we use:
\begin{alignat*}{3}
	&\ \textnormal{- bulk criterion:} & \quad & \text{mark a subset } \mathcal{M} \subset \Ga_h^n \text{ such that } \eta (T) \leq \theta^\star \, \eta_{\max} \quad \forall T \in \mathcal{M}, \\
	&\ \textnormal{- D\"orfler criterion:} & \quad & \text{mark a subset } \mathcal{M} \subset \Ga_h^n \text{ such that } \sum_{T \in \mathcal{M}} \eta (T)^2 \leq \theta^\star \,  \eta_{\Ga_h^n}^2 .
\end{alignat*}

We chose the newest vertex bisection (NVB) as the refinement strategy, which is adapted from the flat case. By lifting the new nodes after each refinement step we ensure that our discrete surfaces stay an interpolation and adapt to the geometry correctly. We adapted the efficient refinement implementations of \cite{Funken2019AdaptiveMR}. For more insight into refinement strategies we refer to \cite{Verfuehrt1996,Funken2019AdaptiveMR}.

The coarsening is non-trivial, especially if the mesh history should not be explicitly stored. We avoid hanging nodes by only selecting ``good to coarsen'' nodes. Details for NVB are given by Chen and Zhang in \cite{Chen2010NVBCoarse}. We adapted their coarsening implementations to surfaces which includes lifting, see \cite{CoarsenNVB} for the Euclidean code.

\subsection{Computing the refinement interpolation}
\label{section:refinement interpolation - implementation}
We give a description on the implementation of the \emph{refinement interpolant}. The main idea was already stated in Section~\ref{sec_sub:ref_int}.

We evaluate the function $u_h^{n-1}$ for new nodes created during refinement. If an element is chosen for refinement we first refine on $\Ga_h^{n-1}$, store the evaluation of $u_h^{n-1}$ at the new nodes and create new subtriangles. Afterwards we lift the new nodes onto $\Ga$ to obtain the new discrete surface $\Ga_h^n$. 
We assume that we stored the evaluations of $u_h^{n-1}$ for all nodes $z_i \in \overline{\mathcal{N}^{n}}$, where $\overline{\mathcal{N}^{n}}$ collects all nodes after refinement, but without lifting. Write those evaluations as $u_{h,i}^{n-1 \rightarrow n}$, notice that the numbering does in general not coincide with the prior indices of $u_{h,i}^{n-1}$. 
Then the interpolation is given as $\IntRef u_h^{n-1} = \sum_{i = 1}^{N^n} u_{h,i}^{n-1 \rightarrow n} \phi_i^n$. This allows us to represent the function $u_h^{n-1}$ on $\Ga_h^n$.

Finally,  the new nodes are numerically lifted onto the surface $\Ga$, this is performed  using DistMesh \cite{distmesh}.
 In a similar fashion we can compute the common interpolations.  

\section{Numerical experiments}
\label{ch:experiments}
We report on numerical experiments illustrating and complementing our theoretical results. The code is  based on the $\ell$FEM Matlab package \cite{ellFEM}, and  is publicly accessible on \texttt{\url{https://git.uni-paderborn.de/lantelme/parabolic-stationary-asfem}}. 
The code was structured according to Algorithm~\ref{alg:base_alg}.

For all experiments, we use the bulk criterion for marking. The initial meshes are generated using DistMesh \cite{distmesh}. 
In our numerical experiments, we use the finite element interpolation instead of the $L^2$-projection. This introduces an error of second order, hence this will not affect the properties of our indicator. 

\subsection{Errors and number of nodes}
\label{sec:Ex1}
We consider the PDE \eqref{eq:heat_strong} on the unit sphere $\Ga = \{x \in \R^3 \mid |x| = 1\}$, 
and we let $f(x,t) = 5 \exp(-t) x_1 x_2$ such that the exact solution is known to be $u(x,t) = \exp(-t) x_1 x_2$.
In Figure~\ref{fig:ex1_l2 - numnodes}~(a) we see the discrete error asymptotic for various tolerances. In Figure~\ref{fig:ex1_l2 - numnodes}~(b) we see the exponential decay in the number of nodes, where we set the tolerance $\TOL = 0.2$, for better illustration $\TOL_{coarse} = 10 \TOL$, and the final time to $T = 10$.

\begin{figure}
    \centering
    \begin{subfigure}{0.48\linewidth}
        \centering
        \resizebox{\linewidth}{\linewidth}{
% [ML]: Update 30.03.26 based on SIAM submission
\begin{tikzpicture}

\begin{axis}[%
width=0.9\linewidth,
height=0.9\linewidth,
at={(0cm,0cm)},
scale only axis,
xmode=log,
xmin=0.05,
xmax=1,
xminorticks=true,
xlabel style={font=\color{white!15!black}},
xlabel={$\TOL$},
ymode=log,
ymin=0.001,
ymax=0.2,
yminorticks=true,
axis background/.style={fill=white},
title={(a) $L^\infty(L^2)$- and $L^2(H^1)$-errors},
legend style={at={(0.03,0.97)}, anchor=north west, legend cell align=left, align=left, draw=white!15!black}
]
\addplot[only marks, mark=o, mark options={}, mark size=2pt, draw=black] table[row sep=crcr]{%
	x	y\\
0.774596669241483	0.0629791918586837\\
0.632455532033676	0.0557957431728085\\
0.447213595499958	0.0344147948286602\\
0.316227766016838	0.0251651772172363\\
0.223606797749979	0.0179596765205889\\
0.187082869338697	0.0154137873175792\\
0.14142135623731	0.0122128066903707\\
0.1	0.00820209861958757\\
0.0707106781186548	0.00630057618458235\\
};
\addlegendentry{$\|e_h^n\|_{L^2(H^1)}$}

\addplot[only marks, mark=x, mark options={}, mark size=3pt, draw=black] table[row sep=crcr]{%
x	y\\
0.774596669241483	0.0157411459526324\\
0.632455532033676	0.0128148493508229\\
0.447213595499958	0.0109770793301243\\
0.316227766016838	0.00793091809002189\\
0.223606797749979	0.0053325550346375\\
0.187082869338697	0.00466325008901014\\
0.14142135623731	0.00382123639216772\\
0.1	0.00256524963422352\\
0.0707106781186548	0.00190805818766581\\
};
\addlegendentry{$\|e_h^n\|_{L^\infty(L^2)}$}

\addplot [color=white!70!black, dashed]
table[row sep=crcr]{%
	0.0707106781186548	0.00544472221513642\\
	0.707106781186548	0.0544472221513642\\
};
\addlegendentry{$\mathcal{O} (\sqrt{TOL})$}

\addplot [color=white!70!black, dashed]
  table[row sep=crcr]{%
0.0707106781186548	0.00164048773235279\\
0.707106781186548	0.0164048773235279\\
};

\end{axis}
\end{tikzpicture}%
}
    \end{subfigure}
    \hfill
    \begin{subfigure}{0.48\linewidth}
        \centering
        \resizebox{\linewidth}{\linewidth}{
\begin{tikzpicture}

\begin{axis}[%
width=0.9\linewidth,
height=0.9\linewidth,
at={(0cm,0cm)},
scale only axis,
xmin=0,
xmax=10,
xlabel style={font=\color{white!15!black}},
xlabel={$t$},
ymode=log,
ymin=800,
ymax=10000,
yminorticks=true,
ylabel style={font=\color{white!15!black}},
ylabel={degrees of freedom},
title={(b) number of nodes},
axis background/.style={fill=white},
legend style={scale=1, at={(northwest)}, anchor=northwest}
]
\addplot[const plot, color=black, line width=1.0pt, forget plot, mark=o] table[row sep=crcr] {%
0	8019\\
0.15625	8019\\
0.3125	7715\\
0.625	7682\\
0.9375	5527\\
1.5625	4478\\
2.8125	2874\\
5.3125	1745\\
10	1079\\
};
\end{axis}
\end{tikzpicture}%
}
    \end{subfigure}
    \caption{(a) $L^\infty(L^2)$- and $L^2(H^1)$-errors for a set of tolerances. (b) Number of nodes over time for an exponentially decaying solution.}
    \label{fig:ex1_l2 - numnodes}
\end{figure}

\subsection{Efficiency and reliability test}
\label{sec:ex4}

To illustrate that the error indicator behaves efficiently and reliably as shown in Theorem~\ref{thm:upper_lower}, we computed the errors and error indicators for the problem in Section~\ref{sec:Ex1} without adaptivity.

In Figure~\ref{fig:equivalence} we report on the errors  between the exact solution and the lifted discrete solution (computed via a sufficiently high-order quadrature rule)   and residual-based estimators (for time-step sizes $\tau = 1, 0.1, 0.01$ over the time interval $[0,1]$, and meshes with various degrees of freedom, see plots). 
For fine meshes the temporal errors dominate, hence the error curves flatten out, while for coarser meshes the spatial error is dominating, and the error curves nicely follow the expected convergence order (see reference lines). 
It is clearly observable that the estimator curves behave like the error curves (roughly a constant multiple of each other), i.e.~the indicator is indeed efficient and reliable  with respect to the graph norm $\|\cdot \|_X$, see \eqref{eq:graph_norm}, in particular observe the curves with the largest time step size $\tau = 1$. We note that, as for the flat case \cite{Verfuerth2003}, the estimator is not shown to be optimal in terms of the $L^\infty(L^2)$-norm. 
\begin{figure}[htbp]
	\centering
	\resizebox{1\linewidth}{0.55\linewidth}{
		% [ML]: Update 27.03.26 based on SIAM submission
\begin{tikzpicture}

\begin{axis}[%
width=3.486in,
height=4.245in,
at={(1.354in,0.573in)},
scale only axis,
xmode=log,
xmin=0.0347065383395735,
xmax=1.633,
xminorticks=true,
xlabel style={font=\color{white!15!black}},
xlabel={$h$},
ymode=log,
ymin=0.0001,
ymax=250,
yminorticks=true,
axis background/.style={fill=white},
title={\large{Estimator and $L^\infty(L^2)$  error for different $\tau$}},
legend style={at={(0.01,0.99)}, anchor=north west, legend cell align=left, align=left, draw=white!15!black},
legend columns=2,
]
\addplot [color=white!40!black, mark=square, mark options={solid, white!40!black}, line width=1pt,mark size=3pt]
  table[row sep=crcr]{%
1.633	0.113943858168359\\
1.01089892691324	0.0985902732399946\\
0.541588507842708	0.020225178551037\\
0.275912111333053	0.0127456412164065\\
0.138617569905748	0.0213570587621034\\
0.0693921825559949	0.0236539296760104\\
0.0347065383395735	0.0242363902591298\\
};
\addlegendentry{$L^\infty(L^2)$ error $(\tau \! =\! 1)$}

\addplot [color=white!70!black, mark=square, mark options={solid, white!70!black}, line width=1pt,mark size=3pt]
  table[row sep=crcr]{%
1.633	6.48535418543423\\
1.01089892691324	4.1818026120648\\
0.541588507842708	2.67796062049484\\
0.275912111333053	1.79499423360922\\
0.138617569905748	1.49709315771932\\
0.0693921825559949	1.39745637588206\\
0.0347065383395735	1.36062871364868\\
};
\addlegendentry{estimator $\eta$ $(\tau \! =\! 1)$}

\addplot [color=white!40!black, mark=o, mark options={solid, white!40!black}, line width=1pt,mark size=3pt]
  table[row sep=crcr]{%
1.633	0.253561634170624\\
1.01089892691324	0.225345833941862\\
0.541588507842708	0.0747780000234978\\
0.275912111333053	0.018652199618602\\
0.138617569905748	0.00324327283278135\\
0.0693921825559949	0.00219250763023456\\
0.0347065383395735	0.0031650775029385\\
};
\addlegendentry{$L^\infty(L^2)$ error $(\tau \! =\! 0.1)$}

\addplot [color=white!70!black, mark=o, mark options={solid, white!70!black}, line width=1pt,mark size=3pt]
  table[row sep=crcr]{%
1.633	9.79963306123265\\
1.01089892691324	5.86107584098899\\
0.541588507842708	3.24914817919574\\
0.275912111333053	1.50187733130747\\
0.138617569905748	0.726042129727878\\
0.0693921825559949	0.378186799373583\\
0.0347065383395735	0.227026867762192\\
};
\addlegendentry{estimator $\eta$ $(\tau \! =\! 0.1)$}

\addplot [color=white!40!black, mark=x, mark options={solid, white!40!black}, line width=1pt,mark size=3pt]
  table[row sep=crcr]{%
1.633	0.275002425269874\\
1.01089892691324	0.241681311438953\\
0.541588507842708	0.0828890522071449\\
0.275912111333053	0.0228194905707763\\
0.138617569905748	0.00562720891868844\\
0.0693921825559949	0.00117254641185743\\
0.0347065383395735	0.000224433434978408\\
};
\addlegendentry{$L^\infty(L^2)$ error $(\tau \! =\! 0.01)$}

\addplot [color=white!70!black, mark=x, mark options={solid, white!70!black}, line width=1pt,mark size=3pt]
  table[row sep=crcr]{%
1.633	10.3382397692153\\
1.01089892691324	6.13319145451689\\
0.541588507842708	3.37726424622123\\
0.275912111333053	1.5529133062863\\
0.138617569905748	0.737698640277863\\
0.0693921825559949	0.360921312612738\\
0.0347065383395735	0.179184161451872\\
};
\addlegendentry{estimator $\eta$ $(\tau \! =\! 0.01)$}

\addplot [color=black, dashed, line width=1pt,mark size=3pt]
  table[row sep=crcr]{%
1.633	1.633\\
1.01089892691324	1.01089892691324\\
0.541588507842708	0.541588507842708\\
0.275912111333053	0.275912111333053\\
0.138617569905748	0.138617569905748\\
0.0693921825559949	0.0693921825559949\\
0.0347065383395735	0.0347065383395735\\
};
\addlegendentry{$\mathcal{O} (h)$}

\addplot [color=black, dotted, line width=1pt,mark size=3pt]
  table[row sep=crcr]{%
	1.633	0.96000804 \\
	0.0347065383395735	0.000361363141054887\\
};
\addlegendentry{$\mathcal{O} (h^2)$}

\end{axis}

\begin{axis}[%
width=3.486in,
height=4.245in,
at={(5.941in,0.573in)},
scale only axis,
xmode=log,
xmin=0.0347065383395735,
xmax=1.633,
xminorticks=true,
xlabel style={font=\color{white!15!black}},
xlabel={$h$},
ymode=log,
ymin=0.01,
ymax=100,
yminorticks=true,
axis background/.style={fill=white},
title={\large{Estimator and $L^2(H^1)$ error for different $\tau$}},
legend style={at={(0.01,0.99)}, anchor=north west, legend cell align=left, align=left, draw=white!15!black},
legend columns=2,
]
\addplot [color=white!30!black, mark=square, mark options={solid, white!30!black}, line width=1pt,mark size=3pt]
  table[row sep=crcr]{%
1.633	0.191226013144185\\
1.01089892691324	0.211690372639386\\
0.541588507842708	0.113333052292769\\
0.275912111333053	0.0682162352306555\\
0.138617569905748	0.0521655115138766\\
0.0693921825559949	0.0474003838319553\\
0.0347065383395735	0.0461359764250296\\
};
\addlegendentry{$L^2(H^1)$ error $(\tau \! =\! 1)$}

\addplot [color=white!70!black, mark=square, mark options={solid, white!70!black}, line width=1pt,mark size=3pt]
  table[row sep=crcr]{%
1.633	6.48535418543423\\
1.01089892691324	4.1818026120648\\
0.541588507842708	2.67796062049484\\
0.275912111333053	1.79499423360922\\
0.138617569905748	1.49709315771932\\
0.0693921825559949	1.39745637588206\\
0.0347065383395735	1.36062871364868\\
};
\addlegendentry{estimator $\eta$ $(\tau \! =\! 1)$}

\addplot [color=white!30!black, mark=o, mark options={solid, white!30!black}, line width=1pt,mark size=3pt]
  table[row sep=crcr]{%
1.633	0.471687317473611\\
1.01089892691324	0.525595223431429\\
0.541588507842708	0.281302086184771\\
0.275912111333053	0.14223897059865\\
0.138617569905748	0.0716641268684162\\
0.0693921825559949	0.0365070761618542\\
0.0347065383395735	0.0194639511501807\\
};
\addlegendentry{$L^2(H^1)$ error $(\tau \! =\! 0.1)$}

\addplot [color=white!70!black, mark=o, mark options={solid, white!70!black}, line width=1pt,mark size=3pt]
  table[row sep=crcr]{%
1.633	9.79963306123265\\
1.01089892691324	5.86107584098899\\
0.541588507842708	3.24914817919574\\
0.275912111333053	1.50187733130747\\
0.138617569905748	0.726042129727878\\
0.0693921825559949	0.378186799373583\\
0.0347065383395735	0.227026867762192\\
};
\addlegendentry{estimator $\eta$ $(\tau \! =\! 0.1)$}

\addplot [color=white!30!black, mark=x, mark options={solid, white!30!black}, line width=1pt,mark size=3pt]
  table[row sep=crcr]{%
1.633	0.504792875301036\\
1.01089892691324	0.559512439960544\\
0.541588507842708	0.299527370500333\\
0.275912111333053	0.151232644303781\\
0.138617569905748	0.0758459156130347\\
0.0693921825559949	0.0379789625875172\\
0.0347065383395735	0.0190119638958577\\
};
\addlegendentry{$L^2(H^1)$ error $(\tau \! =\! 0.01)$}

\addplot [color=white!70!black, mark=x, mark options={solid, white!70!black}, line width=1pt,mark size=3pt]
  table[row sep=crcr]{%
1.633	10.3382397692153\\
1.01089892691324	6.13319145451689\\
0.541588507842708	3.37726424622123\\
0.275912111333053	1.5529133062863\\
0.138617569905748	0.737698640277863\\
0.0693921825559949	0.360921312612738\\
0.0347065383395735	0.179184161451872\\
};
\addlegendentry{estimator $\eta$ $(\tau \! =\! 0.01)$}

\addplot [color=black, dashed, line width=1pt,mark size=3pt]
  table[row sep=crcr]{%
1.633	1.633\\
1.01089892691324	1.01089892691324\\
0.541588507842708	0.541588507842708\\
0.275912111333053	0.275912111333053\\
0.138617569905748	0.138617569905748\\
0.0693921825559949	0.0693921825559949\\
0.0347065383395735	0.0347065383395735\\
};
\addlegendentry{$\mathcal{O} (h)$}

\end{axis}
\end{tikzpicture}%
}
	\caption{Errors and error estimates for various time steps and meshes illustrating efficiency and reliability.}
	\label{fig:equivalence}
\end{figure}
\subsection{A moving peak experiment: refinement and coarsening}
\label{sec:Moving_Peak}
The surface version of the  moving peak experiment from \cite[Section~5.3.1]{AdaptiveAlgHeatEq} is used as a benchmark example. We construct an exponential peak moving along the equator (in the $x_1$--$x_2$-plane) which briefly vanishes at the time $t_0$, for $x = (x_1,x_2,x_3)$:
\begin{equation*}
	u(x,t) = \big( 1-\exp(-b(t-t_0)^2) \big) \, \exp \Big( -a \Big( \Big(x_1-\cos{\tfrac{t}{R} \pi}\Big)^2 + \Big(x_2-\sin{\tfrac{t}{R} \pi}\Big)^2 + x_3^2 \Big) \Big).
\end{equation*}
We then compute the corresponding right-hand side $f$ such that the moving peak $u$ solves the surface heat equation.

The parameters $a = 25$, $b = 200$, and $R = 2$ respectively determine the sharpness of the peak, the speed at which the peak vanishes at $t_0 = 1/2$, and the number of revolutions around the equator. 
Here the peak will move from $(1,0,0)$ to $(0,1,0)$ while vanishing briefly at the midpoint. %$(\frac{\sqrt{2}}{2}, \frac{\sqrt{2}}{2},0)$. 

Figure~\ref{fig:MeshesPeak} reports on the adaptively obtained meshes at different time steps for the moving peak experiment with $\TOL = 2$. For better illustration we forced the temporal indicator to be smaller $\TOL_{\tau} = 0.1 \TOL$ and the coarsening to be larger $\TOL_{coarse} = 10 \TOL$, but the significant decrease of nodes around $t^0$ always occurs.
\begin{figure}[htbp]
    \centering
    \includegraphics[width=1\linewidth]{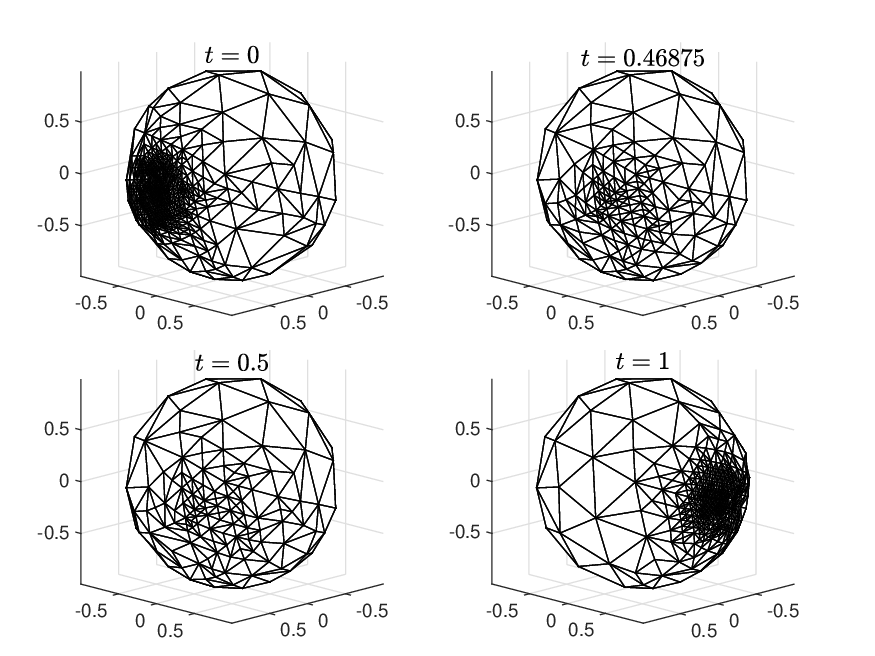} %MeshPeak.pdf
    \caption{Adaptively obtained meshes at different time steps for the moving peak experiment ($\TOL = 2$).}
	\label{fig:MeshesPeak}
\end{figure}

\section*{Acknowledgements}
The work of Bal\'azs Kov\'acs is funded by the Heisenberg Programme of the Deutsche Forschungsgemeinschaft (DFG, German Research Foundation) -- Project-ID 446431602,
and by the DFG Research Unit FOR 3013 \textit{Vector- and tensor-valued surface PDEs} (BA2268/6–1).

\bibliographystyle{siamplain}
\bibliography{literature_siam}

\end{document}